\documentclass[11pt]{amsart}
\usepackage{fullpage}
\usepackage{amssymb}

\usepackage{srcltx,hyperref,tikz}
\usetikzlibrary{shapes,calc,through,backgrounds,patterns,arrows,decorations.pathreplacing}
\usepackage{array}
\newcommand*\circled[1]{\tikz[baseline=(char.base)]{
    \node[shape=diamond,draw,inner sep=0pt] (char) {#1};}}

\newcommand*\thickcircled[1]{\tikz[baseline=(char.base)]{
    \node[shape=diamond,draw,inner sep=0pt,very thick] (char) {#1};}}

\newcommand*\diamonded[1]{\tikz[baseline=(char.base)]{
    \node[shape=circle,draw,inner sep=1pt] (char) {#1};}}

\newcommand*\thickdiamonded[1]{\tikz[baseline=(char.base)]{
    \node[shape=circle,draw,inner sep=1pt,very thick] (char) {#1};}}

\newcommand*\thickcirclediamonded[1]{\tikz[baseline=(char.base)]{
    \node[shape=diamond,draw,inner sep=0pt,very thick] (char) {#1};
    \node[shape=circle,draw,inner sep=1pt,very thick] (char) {#1};}}

\usepackage[enableskew,vcentermath]{youngtab}

\newtheorem{theorem}{Theorem}[section]
\newtheorem{lemma}[theorem]{Lemma}
\newtheorem{cor}[theorem]{Corollary}
\newtheorem{prop}[theorem]{Proposition}
\newtheorem{cnj}[theorem]{Conjecture}

\theoremstyle{definition}
\newtheorem{definition}[theorem]{Definition}

\theoremstyle{remark}
\newtheorem{remark}[theorem]{Remark}

\numberwithin{equation}{section}

\newcommand{\Set}[1]{\ensuremath{\mathcal{#1}}}
\newcommand{\Mat}[1]{\ensuremath{\mathbf{#1}}}
\newcommand{\Dfn}[1]{\emph{#1}}
\newcommand{\naturals}{\mathbb N}
\newcommand{\integers}{\mathbb Z}

\let\size=\abs

\newcommand\ben{\begin{enumerate}}
\newcommand\een{\end{enumerate}}
\newcommand\beq{\begin{equation}}
\newcommand\eeq{\end{equation}}

\newcommand{\F}{(x,y)}
\DeclareMathOperator{\swapall}{\Phi}
\DeclareMathOperator{\swap}{\phi}
\DeclareMathOperator{\switch}{\mu}%
\renewcommand{\b}{{\bf b}}
\renewcommand{\t}{{\bf t}}
\newcommand{\h}{{\bf h}}
\renewcommand{\P}{\Set P}
\newcommand{\PP}{\Set{\widetilde P}}
\def\mn{\mbox{-}}
\renewcommand{\S}{\mathcal S}
\newcommand{\R}{\mathcal P}

\newcommand{\nP}{\hat P}
\newcommand{\B}{{\Set B}_{T,B}}
\newcommand{\BB}{{\Set B}}
\newcommand{\LPM}{{\mathfrak L}_{T,B}}
\DeclareMathOperator{\Tab}{Tab}

\newcommand{\D}{\Set D}
\newcommand{\T}{\Set T}
\def\ec{e_a}
\def\er{e_l}
\def\ecb{e_{\bar a}}%
\def\erb{e_{\bar l}}%
\def\bij{\Psi}

\def\N{-- ++(0,1) circle(1.2pt)}
\def\E{-- ++(1,0) circle(1.2pt)}
\def\n{-- ++(0,1) circle(1.2pt)}
\def\e{-- ++(1,0) circle(1.2pt)}
\def\s{-- ++(0,-1) circle(1.2pt)}

\renewcommand{\labelenumi}{(\roman{enumi})}

\begin{document}

\title{Symmetries of statistics on lattice paths between two boundaries}
\author{Sergi Elizalde}%
\email{sergi.elizade@dartmouth.edu}%
\curraddr{Department of Mathematics, Dartmouth College, Hanover, NH
  03755, USA}%
\thanks{First author partially supported by grant DMS-1001046 from the NSF, 
by grant \#280575 from the Simons Foundation, and by grant H98230-14-1-0125 from the NSA}
\author{Martin Rubey}%
\email{martin.rubey@tuwien.ac.at}%
\curraddr{Institut f\"ur Diskrete Mathematik und Geometrie, TU Wien, Austria}%
\thanks{Second author partially supported by EC's IHRP Programme,
  grant HPRN-CT-2001-00272, \lq\lq Algebraic Combinatorics in Europe\rq\rq\ and the European Research
Council (ERC), Project number 306445}%
\subjclass[2010]{Primary 05A19; Secondary 52C05, 05A15, 05B35, 05E05}
\keywords{lattice path, combinatorial statistic, bijection, triangulation, semistandard Young tableau, flagged tableau, Tutte polynomial, matroid}%

\date{}

\begin{abstract}
We prove that on the set of lattice paths with steps $N=(0,1)$ and $E=(1,0)$ that lie between two fixed boundaries $T$ and $B$ (which are themselves lattice paths), the statistics
`number of $E$ steps shared with $B$' and `number of $E$ steps shared with $T$' have a symmetric joint distribution. To do so, we give an involution that switches these statistics, preserves additional parameters, and
generalizes to paths that  contain steps $S=(0,-1)$ at prescribed $x$-coordinates.
We also show that a similar equidistribution result for path statistics follows from the fact that the Tutte polynomial of a matroid is independent of the order of its ground set.
We extend the two theorems to $k$-tuples of paths between two boundaries, and we give some applications to Dyck paths, generalizing a result of Deutsch, to watermelon configurations, to pattern-avoiding permutations, and to the generalized Tamari lattice.

 Finally, we prove a conjecture of Nicol\'as about the distribution of degrees of $k$ consecutive vertices in $k$-triangulations of a convex $n$-gon. To achieve this goal, we provide a
new statistic-preserving bijection between certain $k$-tuples of non-crossing paths and $k$-flagged semistandard Young tableaux, which is based on local moves reminiscent of {\it jeu de taquin}.
\end{abstract}

\maketitle

\section{Introduction}\label{sec:intro}

Lattice paths have been studied for centuries. Following a recent
survey~\cite{MR2609483}, we can find a drawing of a lattice path
already in a paper from 1878 by William
Whitworth~\cite{Whitworth1878} (in a table between page 128 and 129
of the journal), who used it to describe a solution of the ballot
problem.  Joseph Bertrand rediscovered the problem and its solution
in 1887.  A little later he stresses that, although posed as a
recreational exercise, the problem is important because of its
connection with the gambler's ruin.

Nowadays, a frequently stated motivation to explore lattice paths is the study of polymers in
dilute solution.  Neglecting any interactions among the monomers we
obtain the simplest model of a polymer, a so-called \lq ideal
chain\rq, which is just a random walk.  Being a little more
realistic, one would not allow several monomers to share a single
location, and thus consider self-avoiding walks,
see~\cite{Vanderzande1998}.  However, this model is
mathematically hardly tractable, which possibly explains why it
remains such a great source of inspiration for combinatorialists.  In
particular, one of the main problems ---without any satisfying
solution in sight--- is to determine the number of paths of given
length, at least approximately.

A simpler model, still interesting from the physicists' point of view,
is that of directed lattice paths: choose a preferred
direction on the lattice and allow only steps whose
projection onto this direction is non-negative.  For example, one could consider
(self-avoiding) paths with unit north, east and south steps only.  In
contrast to general self-avoiding walks, many variations of this
model are now well understood.  In particular, the problem of
enumerating such paths confined to a prescribed subset of
$\integers^2$, such as above the main diagonal or inside a wedge,
has been successfully tackled.  Another important variation of the
problem was initiated by Michael Fisher~\cite{Fisher}, who considered
families of non-intersecting paths with unit north and east steps,
so-called \lq watermelons\rq, to model phase transitions in wetting
and melting processes.  It is then also of interest to refine the
enumeration by taking into account the number of contacts of the path
with the boundary.

We do not have to resort to physics to find applications of lattice
paths, especially when allowing only unit north and east steps.  In
particular, questions involving Dyck paths are ubiquitous in
algebraic combinatorics.  As a striking example, let us mention that the
bigraded Hilbert series of the space of alternating diagonal
harmonics is given by the $q,t$-Catalan polynomials~\cite{MR2371044}.  
Alternatively, these can be described as the
generating polynomials for two natural statistics on Dyck paths, \lq
area\rq\ and \lq bounce\rq, or equivalently, \lq dinv\rq\ and \lq area\rq.
Although it is possible to show that the
polynomials are symmetric in $q$ and $t$, no bijective proof of this
innocent-looking fact is known.
The number of contacts with the boundary also plays a role in this setting.
Jim Haglund \cite[p. 50]{MR2371044} gives a bijection $\zeta$ on Dyck paths, 
which had been used in a different context by George Andrews {\it et al.}~\cite{MR1926854},
sending the pair $(\mathrm{area},\mathrm{dinv})$ to the pair $(\mathrm{bounce},\mathrm{area})$.
As pointed out by Christian Stump, the bijection $\zeta$ also sends the statistic `number of returns'
(i.e., contacts with the diagonal)
to the length of the initial rise (equivalently, height of the first peak, or contacts with the left boundary).
More generally, Nick Loehr's extension~\cite{MR2134172} of this map to $m$-Dyck paths has the analogous property.

In fact, it was shown by Emeric Deutsch~\cite{Deutsch1998, Deutsch1999} that the joint distribution of the pair of statistics
`number of returns' and `height of the first peak' is symmetric on
Dyck paths, by exhibiting a recursively defined involution. 
A non-recursive description of the same map has been given by Mireille Bousquet-M\'elou, \'Eric Fusy, and Louis-Fran\c{c}ois Pr\'eville-Ratelle~\cite{BousquetMelou}.

Our contribution to the study of directed lattice paths confined to a region is twofold.  
On the one hand, in Section~\ref{sec:bottom-top}
we generalize and refine the above theorem of Deutsch. We show that the
restriction to Dyck paths is not necessary: there are natural
generalizations of the statistics to paths in an arbitrary region
where the symmetry still holds.  Moreover, it is possible to extend the
result to paths taking east, north and south steps.

Focusing on paths with east and north steps only, in
Section~\ref{sec:bottom-left-top-right} we give a different
generalization of the two statistics considered by Deutsch to two
pairs of statistics which are equidistributed over the set of paths
in a region.  In Section~\ref{sec:path-tuple} we generalize both
results to $k$-tuples of non-intersecting paths.  Finally, in
Section~\ref{sec:applications} we investigate some consequences of
these theorems to Dyck paths, watermelon configurations, and
restricted permutations.  We also establish a link between our map
and the covering relation in the generalized Tamari lattices
recently introduced by Fran\c{c}ois Bergeron~\cite{Bergeron2011}.

On the other hand, in Section~\ref{sec:SSYT} we provide a new
connection between two seemingly separate topics in algebraic
combinatorics: we exhibit a natural weight-preserving bijection
between $k$-flagged semistandard Young tableaux of shape $\lambda$,
as appearing when studying Schubert polynomials~\cite{Wachs1985}, and
families of $k$ non-intersecting paths confined to a region of the
same shape $\lambda$.  We use this bijection to establish a
refinement of a conjecture on multi-triangulations of polygons by
Carlos Nicol\'as~\cite{Nicolas2009}, concerning the distribution of
degrees of $k$ consecutive vertices.

\section{Statement of main results}
\label{sec:statement}

Let $T$ and $B$ be two lattice paths in $\naturals^2$ with north
steps ($N=(0,1)$) and east steps ($E=(1,0)$) from the origin to some
prescribed point $\F\in\naturals^2$ such that $T$ is weakly above
$B$, i.e., the $n$-th east step of $T$ is weakly above the $n$-th
east step of $B$ for $1\leq n\leq x$.  Let $\P(T, B)$ be the set of
lattice paths with north and east steps from the origin to $\F$ that
lie between $T$ and $B$, i.e., weakly above $B$ and weakly below $T$.
Thus, the paths $T$ and $B$ are the \Dfn{upper} and \Dfn{lower}
boundaries of the paths in $\P(T, B)$.

In this paper we show that several natural statistics on lattice paths
in $\P(T, B)$ have a symmetric distribution.  Formally, a
\Dfn{statistic} on a set $\mathcal{O}$ of objects is simply a function from
$\mathcal{O}$ to $\naturals$.  Two $k$-tuples of statistics
$(f_1,f_2,\dots,f_k)$ and $(g_1,g_2,\dots,g_k)$ have the same
\Dfn{joint distribution} over $\mathcal{O}$, denoted
$$
(f_1,f_2,\dots,f_k)\sim(g_1,g_2,\dots,g_k),
$$
if
$$
\sum_{P\in\mathcal{O}} x_1^{f_1(P)}\dots x_k^{f_k(P)} = \sum_{P\in\mathcal{O}} x_1^{g_1(P)}\dots x_k^{g_k(P)}.
$$
The distribution of $(f_1,f_2,\dots,f_k)$ is \Dfn{symmetric} over
$\mathcal{O}$ if
$$
(f_1,f_2,\dots,f_k)\sim(f_{\pi(1)},f_{\pi(2)},\dots,f_{\pi(k)})
$$
for every permutation $\pi$ of $[k]=\{1,2,\dots,k\}$.

We consider statistics counting
the following special steps of paths in $\P(T,B)$:
\begin{itemize}
\item a \Dfn{top contact} is an east step that is also a step of $T$,
\item a \Dfn{bottom contact} is an east step that is also a step of $B$,
\item a \Dfn{left contact} is a north step that is also a step of $T$,
\item a \Dfn{right contact} is a north step that is also a step of $B$.
\end{itemize}
We denote the number of top, bottom, left and right contacts of
$P\in\P(T,B)$ by $t(P)$, $b(P)$, $\ell(P)$ and $r(P)$, respectively.  An example
is given in Figure~\ref{fig:paths}.

\begin{figure}[htb]
  \begin{center}
\begin{tikzpicture}[scale=0.65]
 \draw (0,0) circle(1.2pt) \N\N\E\N\E\E\N\E\N\E\N\E\N\E\E\E\E;
 \draw (0,0) circle(1.2pt) \E\E\E\N\E\N\E\E\N\E\N\N\E\E\N\E\N;
 \draw[very thick] (0,0) \e\n\n\n\e\e\e\n\n\e\e\e\e\e\n\n\e;
 \draw[very thick,dotted,red] (1,3) -- (3,3);
 \draw[very thick,dotted,red] (4,5) -- (5,5);
 \draw[very thick,dotted,red] (9,7) -- (10,7);
 \draw[very thick,dotted,blue] (0,0) -- (1,0);
 \draw[very thick,dotted,blue] (7,5) -- (9,5);
 \draw[very thick,dotted,brown] (1,2) -- (1,3);
 \draw[very thick,dotted,brown] (4,4) -- (4,5);
 \draw[very thick,dotted,cyan] (9,5) -- (9,6);
 \draw (1.5,1) node {$P$};
 \draw (5.1,6.5) node {$T$};
 \draw (5,1.5) node {$B$};
 \draw (-0.6,0) node {$O$};
 \draw (10.5,7.1) node {$F$};
\end{tikzpicture}
  \end{center}
  \caption{A path $P\in\P(T,B)$ with  $t(P)=4$, $b(P)=3$, $\ell(P)=2$,
    and $r(P)=1$.}\label{fig:paths}
\end{figure}

In the next two sections we give bijective proofs of the following results.
\begin{theorem}\label{thm:bottom-top}
  The distribution of the pair $(t, b)$ over $\P(T, B)$ is symmetric.
\end{theorem}

\begin{theorem}\label{thm:bottom-left-top-right}
  The pairs $(b, \ell)$ and $(t, r)$ have the same joint distribution over
  $\P(T, B)$.
\end{theorem}

As an example, if $T=NNENEE$ and $B=ENEENN$, then
$$\sum_{P\in\P(T,B)}x^{t(P)}y^{b(P)}=x^3+x^2y+xy^2+y^3+2x^2+2xy+2y^2+2x+2y+1
=\sum_{P\in\P(T,B)}x^{b(P)}y^{t(P)}$$
and
$$\sum_{P\in\P(T,B)}x^{b(P)}y^{\ell(P)}=x^3+x^2y+y^3+2x^2+3xy+3y^2+2x+2y=\sum_{P\in\P(T,B)}x^{t(P)}y^{r(P)}.$$

Although the two theorems look very similar, we do not know
of a uniform proof for them. Let us point out that it is not true that
$(t,b,\ell)\sim(b,t,r)$ in general.  In Section~\ref{sec:bottom-top}
we exhibit an involution proving a refined and generalized
version of Theorem~\ref{thm:bottom-top}. The refinement consists of
keeping track of the sequence of $y$-coordinates of the east steps
that are not contacts, while the generalization consists of allowing
the paths to have south steps at prescribed $x$-coordinates.  The
analogous refinement and generalization of
Theorem~\ref{thm:bottom-left-top-right} does not hold,
and in fact our proof of Theorem~\ref{thm:bottom-top}, given
in Section~\ref{sec:bottom-left-top-right}, is 
very different from that of Theorem~\ref{thm:bottom-left-top-right}.
Namely, we show that both
$$
\sum_{P\in\P(T,B)} x^{b(P)}y^{\ell(P)}
\quad \mathrm{and} \quad
\sum_{P\in\P(T,B)} x^{t(P)}y^{r(P)}
$$
can be interpreted as the Tutte polynomial of the lattice
path matroid associated with $\P(T, B)$, as defined in~\cite{BoninDeMierNoy2003}. To do so, we use the
definition of the Tutte polynomial in terms of \Dfn{activities}, which
relies on a linear ordering on the ground set of the matroid.  The
independence of the Tutte polynomial of this ordering then implies
Theorem~\ref{thm:bottom-left-top-right}.  In fact, as pointed out to us by
Olivier Bernardi, William Tutte's
proof~\cite{MR0061366} of the well-definedness of his dichromate
(essentially the Tutte polynomial restricted to graphs) is almost
bijective.  We make this explicit in Appendix~\ref{sec:Tutte-independence}.

In Section~\ref{sec:path-tuple}, Theorems~\ref{thm:bottom-top} and~\ref{thm:bottom-left-top-right} are generalized to $k$-tuples of non-crossing
paths. Let $\P^k(T, B)$ be the set of $k$-tuples $\Mat P=(P_1,P_2,\dots,P_k)$ such that
$P_i\in\P(T,B)$ for all $i$, and $P_i$ is weakly above $P_{i+1}$ for $1\leq i\leq k-1$.
Let $P_0 = T$ and $P_{k+1} = B$. For $0\leq i\leq k$, denote by $h_i=h_i(\Mat P)$
the number of east steps where $P_i$ and $P_{i+1}$ coincide.
We provide a bijective proof of a generalization of
Theorem~\ref{thm:bottom-top}, which can be stated in a simplified
form as follows.
\begin{theorem}\label{thm:bottom-top-k}
  The distribution of $(h_0,h_1,\dots,h_k)$ over $\P^k(T, B)$ is symmetric.
\end{theorem}
To generalize Theorem~\ref{thm:bottom-left-top-right}, define the
\Dfn{top contacts of $\Mat P$} to be the top contacts of $P_1$, and
denote their number by $t(\Mat P)$.  Similarly, let
$\ell(\Mat P)$ be the number of left contacts of $P_1$, and let
$b(\Mat P)$ (respectively $r(\Mat P)$) be the number of bottom
(respectively right) contacts of $P_k$.
Note that $t(\Mat P)= h_0(\Mat P)$ and $b(\Mat P)=h_k(\Mat P)$ by definition.

\begin{theorem}\label{thm:bottom-left-top-right-k}
  The pairs $(b, \ell)$ and $(t, r)$ have the same joint distribution over
  $\P^k(T, B)$.
\end{theorem}

The above theorems have consequences for the enumeration of generalized Dyck paths with a given number of top and bottom contacts, as well as to the distribution of statistics on restricted permutations.
These are discussed in Section~\ref{sec:applications}, together with a new proof of a result of Richard Brak and John Essam~\cite{BrakEssam2001} on watermelon configurations.

In the special case that $T=N^{n-2k-1} E^{n-2k-1}$ and $B=(E N)^{n-2k-1}$, the elements of
$\P^k(T, B)$ are $k$-fans of Dyck paths of semilength $n-2k$. We denote this set by $\D_n^k$.
The number of such $k$-fans was shown by Jakob Jonsson~\cite{Jonsson2005}
to be equal to the number of $k$-triangulations of a convex $n$-gon,
that is, maximal sets of diagonals such that no $k+1$ of them cross
mutually.  Nicol\'as~\cite{Nicolas2009} made the following conjecture, based on experimental evidence.

\begin{cnj}[\cite{Nicolas2009}]\label{conj:carlos}
The distribution of degrees of $k$ consecutive vertices over the set of $k$-triangulations of a convex $n$-gon equals the distribution of the tuple $(h_0,h_1,\dots,h_{k-1})$ over $\D_n^k$.
\end{cnj}

One of the main results of Section~\ref{sec:SSYT} is a proof of this conjecture. In fact, we show that it can be extended to give a description of the distribution of the full degree sequence of the $k$-triangulation.
One ingredient in our proof is a bijection of Serrano and Stump~\cite{SerranoStump2010} between $k$-triangulations of a convex $n$-gon and $\D_n^k$. The other ingredient, whose proof occupies most of Section~\ref{sec:SSYT}, can be stated in a simplified form as follows.

\begin{theorem}
  Let $T=N^y E^x$, and let $B$ be any path from the origin to $(x,y)$, weakly below $T$ and ending with a north step.
  Let $\lambda$ be the Young diagram bounded by $T$ and $B$. There is an explicit bijection between
  $k$-tuples of paths $\Mat P\in\P^k(T, B)$ and $k$-flagged
  SSYT of shape $\lambda$ such that the number of entries equal to $i+1$ in the tableau is $\lambda_1-h_i(\Mat P)$, for $0\le i\le k$.
\end{theorem}

\section{The symmetry $(t,b)\sim(b,t)$ for a single path}
\label{sec:bottom-top}
In this section we construct an involution $\swapall$ that
proves a generalized version of Theorem~\ref{thm:bottom-top}. It not only applies to
a more general set of paths, but it also gives a refined result by preserving the sequence of
$y$-coordinates of the east steps that are not contacts.

Let $\PP(T, B)$ be the set of lattice paths from the origin to $\F$ with north, east
and south ($S=(0,-1)$) steps, lying weakly below $T$ and weakly above $B$.
Given such a lattice path $P$, the \Dfn{descent set} of $P$ is the
set of $x$-coordinates where south steps occur.  For a fixed subset $\Set D\subset\naturals$, denote by $\PP(T, B, \Set D)$ the set of
paths $P\in\PP(T, B)$ having descent set $\Set D$. Note that $\PP(T, B, \emptyset)=\P(T,B)$ by definition.
The definitions of top and bottom contacts generalize trivially to paths in $\PP(T, B)$. An example is given in Figure~\ref{fig:pathwithS}.

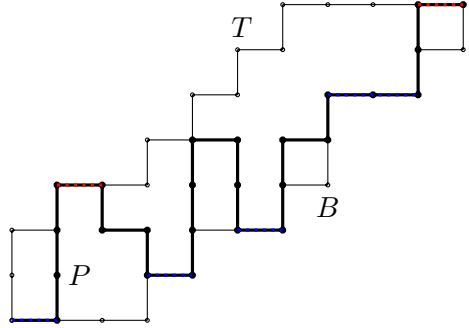
\begin{figure}[htb]
  \begin{center}
\begin{tikzpicture}[scale=0.6]
 \draw (0,0) circle(1.2pt) \N\N\E\N\E\E\N\E\N\E\N\E\N\E\E\E\E;
 \draw (0,0) circle(1.2pt) \E\E\E\N\E\N\E\E\N\E\N\N\E\E\N\E\N;
 \draw[very thick] (0,0) \e\n\n\n\e\s\e\s\e\n\n\n\e\s\s\e\n\n\e\n\e\e\n\n\e;
 \draw[very thick,dotted,red] (1,3) -- (2,3);
 \draw[very thick,dotted,red] (9,7) -- (10,7);
 \draw[very thick,dotted,blue] (0,0) -- (1,0);
 \draw[very thick,dotted,blue] (3,1) -- (4,1);
 \draw[very thick,dotted,blue] (5,2) -- (6,2);
 \draw[very thick,dotted,blue] (7,5) -- (9,5);
 \draw (1.5,1) node {$P$};
 \draw (5.1,6.5) node {$T$};
 \draw (7,2.5) node {$B$};
\end{tikzpicture}
  \end{center}
  \caption{A path $P\in\PP(T,B)$ with  $t(P)=2$, $b(P)=5$, and descent set $\{2,3,5\}$.}\label{fig:pathwithS}
\end{figure}

For a given sequence $\Mat H$ of integers,
let $\PP(T, B, \Set D, \Mat H)$ (respectively $\P(T,B,\Mat H)$) be the subset of
$\PP(T, B, \Set D)$ (respectively $\P(T,B)$) containing those paths whose sequence of
$y$-coordinates of the east steps that are not bottom or top contacts
equals $\Mat H$. Figure~\ref{fig:complete-example} shows all the paths in the set $\PP(T, B, \Set D, \Mat H)$ with  $T=NNNEEENEE$, $B=EENEEENNN$, $\Set D=\{2\}$, and $\Mat H= 2\,2\,3$.

We can now state the announced refinement of
Theorem~\ref{thm:bottom-top}.
\begin{theorem}\label{thm:bottom-top-refined}
  For any set $\Set D$ and any sequence $\Mat H$ of integers,  the
  distribution of $(t, b)$ over $\PP(T, B, \Set D, \Mat H)$ is
  symmetric.
\end{theorem}
\begin{remark}
  Without the refinement by $\Mat H$, a recursive, non-bijective proof of this result has been found independently by Guo Niu Han~\cite{Han}.
\end{remark}
In the following we encode a path in $\PP(T,B)$ by the sequence of
$y$-coordinates of its east steps, except that we record top contacts using $\t$'s and bottom
contacts using $\b$'s. Steps that are simultaneously top and bottom contacts are
not recorded in this encoding at all. For example, the first path in Figure~\ref{fig:complete-example} is encoded by $2\t23\t$.

To prove Theorem~\ref{thm:bottom-top-refined}, we construct an involution $\swapall$ that essentially
turns top contacts into bottom contacts one at a time. The transformation that turns a top contact into a bottom contact, which we denote $\swap$,
relies in turn on a transformation on words, denoted $\switch$, which turns a sequence of $e$ $\t$'s and $f$ $\b$'s into a sequence of $e-1$ $\t$'s and $f+1$ $\b$'s.

\subsection{A transformation on words}

Let us first describe the map $\switch$, which is defined on words over the alphabet $\{\t,\b\}$.
We say that such a word $w_1w_2\dots w_{2n}$ of even length is a Dyck word
if it contains the same number of $\t$'s and $\b$'s, and in every prefix $w_1w_2\dots w_i$ with $1\le i\le 2n$, the
number of $\b$'s never exceeds the number of $\t$'s.

  Let $\Mat w=w_1w_2\dots w_{e+f}$ be a word over the alphabet $\{\t,\b\}$.
  Any such $\Mat w$ can be factorized uniquely as
  \begin{equation}\label{factorization}
  \Mat w = D_1 \b D_2 \b \dots \b D_j \t D_{j+1} \t D_{j+2} \t \dots \t D_m,
  \end{equation}
  where each $D_i$ for $1\le i\le m$ is a (possibly empty) Dyck
  word. In such a factorization, the letters $\t$ and $\b$ which
  are not part of a Dyck word are called \Dfn{unmatched}
  letters. By construction, all the unmatched $\b$'s are to the left of the unmatched $\t$'s.
  Suppose that there is at least one unmatched $\t$.
  We define $\switch(\Mat w)$ to be the word obtained from $\Mat w$ by
  replacing the leftmost unmatched $\t$ with a $\b$, that is,
  \begin{equation}\label{factorization_switch}
  \switch(\Mat w) = D_1 \b D_2 \b \dots \b D_j \b D_{j+1} \t D_{j+2} \t \dots \t D_m.
  \end{equation}

The above factorization can be visualized by representing a word with
a path, drawing an northeast step $(1,1)$ for each letter $\t$ and a southeast step $(1,-1)$ for
each letter $\b$.  In Figure~\ref{fig:switch}, the effect of
the map $\switch$ applied to the word
$\b\t\t\b\t\b\b\b\t\t\b\t\t\b\t\b\t\t$ is shown.
In this example, the Dyck words $D_2$, $D_4$ and $D_5$, indicated by the dotted areas, are non-empty.

\begin{figure}[htb]
  \begin{center}
    \begin{tikzpicture}[scale=0.4]
      \draw (0,5) coordinate(d0)
      -- ++(1,-1) coordinate(d1)
      -- ++(1,+1) coordinate(d2)
      -- ++(1,+1) coordinate(d3)
      -- ++(1,-1) coordinate(d4)
      -- ++(1,+1) coordinate(d5)
      -- ++(1,-1) coordinate(d6)
      -- ++(1,-1) coordinate(d7)
      -- ++(1,-1) coordinate(d8)
      -- ++(1,+1) coordinate(d9) 
      -- ++(1,+1) coordinate(d10)
      -- ++(1,-1) coordinate(d11)
      -- ++(1,+1) coordinate(d12)
      -- ++(1,+1) coordinate(d13)
      -- ++(1,-1) coordinate(d14)
      -- ++(1,+1) coordinate(d15)
      -- ++(1,-1) coordinate(d16)
      -- ++(1,+1) coordinate(d17)
      -- ++(1,+1) coordinate(d18);
      \foreach \x in {0,...,18} {
        \draw (d\x) circle (1.5pt);
      }
      \fill[pattern=dots] (d1) \foreach \x in {2,...,7} { -- (d\x) };
      \fill[pattern=dots] (d9) \foreach \x in {10,11} { -- (d\x) };
      \fill[pattern=dots] (d12) \foreach \x in {13,...,16} { -- (d\x) };
      \draw[red, very thick] (d8)--(d9);
      \draw (d9) ++(0,-2) node [label=$\switch$,rotate=-90] {$\mapsto$};
      \draw (0,0) coordinate(d0)
      -- ++(1,-1) coordinate(d1)
      -- ++(1,+1) coordinate(d2)
      -- ++(1,+1) coordinate(d3)
      -- ++(1,-1) coordinate(d4)
      -- ++(1,+1) coordinate(d5)
      -- ++(1,-1) coordinate(d6)
      -- ++(1,-1) coordinate(d7)
      -- ++(1,-1) coordinate(d8)
      -- ++(1,-1) coordinate(d9) 
      -- ++(1,+1) coordinate(d10)
      -- ++(1,-1) coordinate(d11)
      -- ++(1,+1) coordinate(d12)
      -- ++(1,+1) coordinate(d13)
      -- ++(1,-1) coordinate(d14)
      -- ++(1,+1) coordinate(d15)
      -- ++(1,-1) coordinate(d16)
      -- ++(1,+1) coordinate(d17)
      -- ++(1,+1) coordinate(d18);
      \foreach \x in {0,...,18} {
        \draw (d\x) circle (1.5pt);
      }
      \fill[pattern=dots] (d1) \foreach \x in {2,...,7} { -- (d\x) };
      \fill[pattern=dots] (d9) \foreach \x in {10,11} { -- (d\x) };
      \fill[pattern=dots] (d12) \foreach \x in {13,...,16} { -- (d\x) };
      \draw[blue, very thick] (d8)--(d9);
    \end{tikzpicture}
  \end{center}
  \caption{A visual description of the map $\switch$.}\label{fig:switch}
\end{figure}

\begin{remark}
The map $\switch$, in different contexts, belongs to mathematical folklore. Curtis Greene and Daniel Kleitman~\cite{GreeneKleitman} used
it to build a symmetric chain decomposition of the boolean algebra, giving an
injection from $i$-element subsets of $[n]$ to $(i+1)$-element subsets of $[n]$ (where $i<n/2$) with the property that
each subset is contained in its image. Their construction proves that the boolean algebra has the Sperner property.
It also yields a bijection between $i$-element subsets and $(n-i)$-element subsets of $[n]$ where each subset is contained in its image.
On the other hand, a bijective argument based on iterates of $\switch$ proves that the number of ballot paths of length $2n$ is $\binom{2n}{n}$.
\end{remark}

\begin{lemma}\label{lem:switch1}
  Let $e,f,u$ be nonnegative integers with $u\ge\max\{e-f,f-e+2\}$.
The map $\switch$ is a bijection between
  \begin{enumerate}
  \item the set of words with $e$ $\t$'s and $f$ $\b$'s having $u$ unmatched letters, and
  \item the set of words with $e-1$ $\t$'s and $f+1$ $\b$'s having $u$ unmatched letters.
  \end{enumerate}
\end{lemma}

\begin{proof}
  Let $\Mat w$ be a word with $e$ $\t$'s and $f$ $\b$'s having
  $u$ unmatched letters.  Note that $\Mat w$ has some unmatched
  $\t$, since otherwise we would have $u=f-e < f-e+2$,
  contradicting the assumption on $u$.
  If the factorization of $\Mat w$ is given by~\eqref{factorization},
  then the factorization of $\switch(\Mat w)$ is given by~\eqref{factorization_switch}.
  In particular, $\switch(\Mat w)$ is a word with $e-1$ $\t$'s and
  $f+1$ $\b$'s having the same number of unmatched letters as $\Mat
  w$.  The map is injective because $\Mat w$ can be recovered from
  $\switch(\Mat w)$ by replacing the rightmost unmatched $\b$ with a
  $\t$, and it is surjective because any word with $e-1$ $\t$'s and $f+1$ $\b$'s
  having $u$ unmatched letters must have some
  unmatched $\b$, since otherwise we would have $u=(e-1)-(f+1) < e-f$,
  contradicting the assumption on $u$.
\end{proof}

The next lemma states that the left-most unmatched $\t$ in $\Mat w$
can only be preceded by a $\b$ and followed by a $\t$ in both $\Mat w$ and $\switch(\Mat w)$.

\begin{lemma}\label{lem:switch}
  Let $\Mat w=w_1w_2\dots w_{e+f}$ be a word with $e$ $\t$'s and $f$ $\b$'s having some unmatched $\t$.
  Let $\switch(\Mat w)=w'_1 w'_2\dots w'_{e+f}$, and let $i$ be such that $w_i = \t$ and $w'_i = \b$.
  If $i>1$, then $w_{i-1} = w'_{i-1} =\b$, and if $i<e+f$, then $w_{i+1} =w'_{i+1} = \t$.
\end{lemma}
\begin{proof}
This follows trivially from the definition of $\switch$.
\end{proof}

\subsection{The maps $\swap$ and $\swapall$}

Our next goal is to translate $\switch$, which is a map on words, into a transformation on paths, denoted $\swap$. Then, $\swapall$ will be constructed by iterating $\swap$.

\begin{definition}
  For $P\in\PP(T, B)$, the \Dfn{sequence of contacts} of $P$ is the word $\Mat w_P$ over $\{\t,\b\}$ obtained by recording the top and bottom
  contacts of $P$ from left to right, except for the steps that are simultaneously top and a bottom contacts, which are not recorded.
\end{definition}

\begin{definition}\label{def:swap}
  Let $P\in\PP(T, B)$ be such that $\Mat w_P$ contains some unmatched $\t$.
  The east steps of $P$ can be decomposed uniquely as $P=W X \t Y Z$, where
  \begin{itemize}
  \item the selected $\t$ is the leftmost unmatched $\t$ in $\Mat w_P$,
  \item $X$ is maximal such that there is no descent after
    any of its steps and no (right) endpoint of any of its
    steps lies on $B$, and
  \item $Y$ is maximal such that there is a descent before
    each of its steps.
  \end{itemize}
  Let $h_X$ (respectively $h_Y$) be $-\infty$ if $X$ (respectively $Y$) is empty,
  and otherwise the $y$-coordinate of its last (respectively first)
  east step.  Define
  $$
  \swap(P)=
  \begin{cases}
    W X Y \b Z & \text{if $h_X\le h_Y$,} \\ W
    \b X Y Z & \text{if $h_X> h_Y$.}
  \end{cases}
  $$
\end{definition}

\begin{remark}
  In the case of paths with no descents, the definition of $\swap$
  is simpler: writing $P$ as $P=W X\t Z$, where $X$ is maximal not
  touching $B$, we have $\swap(P)=W\b XZ$.
\end{remark}

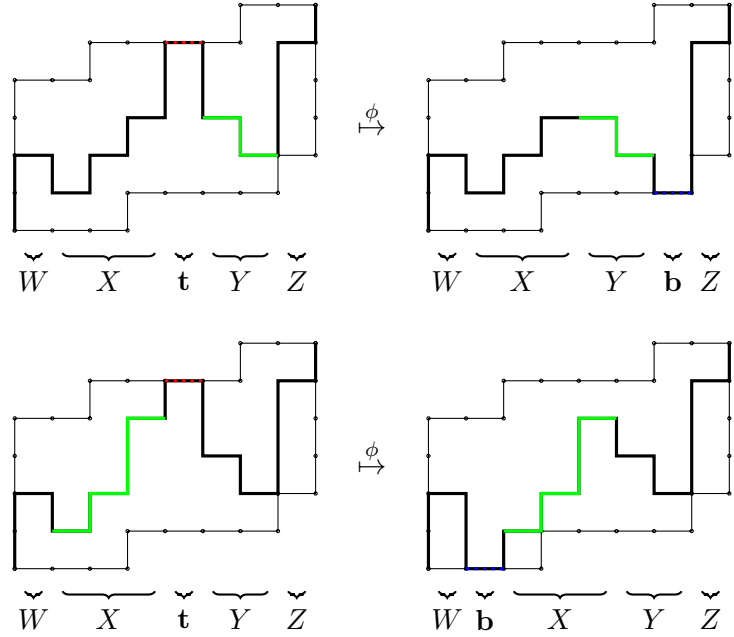
\begin{figure}[hbt]
  \begin{center}
    \def\drawbrace#1#2#3{ 
      \draw[thick, decoration={brace, mirror, raise=0.3cm}, decorate]
      (#1) -- (#2)
      node [pos=0.5,anchor=north,yshift=-0.4cm] {#3};
    }
    \begin{tikzpicture}[scale=0.5]
      \def\R#1{
        \coordinate (origin) at #1;
        \draw (origin) circle(1.2pt) \E\E\E\N\E\E\E\E\N\E\N\N\N\N;
        \draw (origin) circle(1.2pt) \N\N\N\N\E\E\N\E\E\E\E\N\E\E;
      }
      \R{(0,9)}
      \draw[very thick] (origin)
      -- ++(0, 2) -- ++(1, 0) 
      -- ++(0,-1) -- ++(1, 0) 
      -- ++(0, 1) -- ++(1, 0) 
      -- ++(0, 1) -- ++(1, 0) 
      -- ++(0, 2) -- ++(1, 0) 
      -- ++(0,-2) -- ++(1, 0) 
      -- ++(0,-1) -- ++(1, 0) 
      -- ++(0, 3) -- ++(1, 0) 
      -- ++(0, 1);
      \path (origin) node(x0) {}
      ++(1,0) node(x1) {}
      ++(3,0) node(x2) {}
      ++(1,0) node(x3) {}
      ++(2,0) node(x4) {}
      ++(1,0) node(x5) {};
      \drawbrace{x0}{x1}{$W$}
      \drawbrace{x1}{x2}{$X$}
      \drawbrace{x2}{x3}{$\t$}
      \drawbrace{x3}{x4}{$Y$}
      \drawbrace{x4}{x5}{$Z$}
      \draw[very thick,dotted,red] (origin) ++(4,5) -- ++(1,0);
      \draw[very thick,green] (origin) ++(5,3) -- ++(1,0)-- ++(0,-1)-- ++(1,0);
      \draw (origin) ++(9.5,3) node {$\stackrel{\swap}{\mapsto}$};
      \R{(11,9)}
      \draw[very thick] (origin)
      -- ++(0, 2) -- ++(1, 0)
      -- ++(0,-1) -- ++(1, 0)
      -- ++(0, 1) -- ++(1, 0)
      -- ++(0, 1) -- ++(1, 0)
      -- ++(0, 0) -- ++(1, 0)
      -- ++(0,-1) -- ++(1, 0)
      -- ++(0,-1) -- ++(1, 0)
      -- ++(0, 4) -- ++(1, 0)
      -- ++(0, 1);
      \path (origin) node(x0) {}
      ++(1,0) node(x1) {}
      ++(3,0) node(x2) {}
      ++(2,0) node(x3) {}
      ++(1,0) node(x4) {}
      ++(1,0) node(x5) {};
      \drawbrace{x0}{x1}{$W$}
      \drawbrace{x1}{x2}{$X$}
      \drawbrace{x2}{x3}{$Y$}
      \drawbrace{x3}{x4}{$\b$}
      \drawbrace{x4}{x5}{$Z$}
      \draw[very thick,dotted,blue] (origin) ++(6,1) -- ++(1,0);
      \draw[very thick,green] (origin) ++(4,3) -- ++(1,0)-- ++(0,-1)-- ++(1,0);
      \R{(0,0)}
      \draw[very thick] (origin)
      -- ++(0, 2) -- ++(1, 0)
      -- ++(0,-1) -- ++(1, 0)
      -- ++(0, 1) -- ++(1, 0)
      -- ++(0, 2) -- ++(1, 0)
      -- ++(0, 1) -- ++(1, 0)
      -- ++(0,-2) -- ++(1, 0)
      -- ++(0,-1) -- ++(1, 0)
      -- ++(0, 3) -- ++(1, 0)
      -- ++(0, 1);
      \path (origin) node(x0) {}
      ++(1,0) node(x1) {}
      ++(3,0) node(x2) {}
      ++(1,0) node(x3) {}
      ++(2,0) node(x4) {}
      ++(1,0) node(x5) {};
      \drawbrace{x0}{x1}{$W$}
      \drawbrace{x1}{x2}{$X$}
      \drawbrace{x2}{x3}{$\t$}
      \drawbrace{x3}{x4}{$Y$}
      \drawbrace{x4}{x5}{$Z$}
      \draw[very thick,dotted,red] (origin) ++(4,5) -- ++(1,0);
      \draw[very thick,green] (origin) ++(1,1) -- ++(1,0)-- ++(0,1)-- ++(1,0)-- ++(0,2)-- ++(1,0);
      \draw (origin) ++(9.5,3) node {$\stackrel{\swap}{\mapsto}$};
      \R{(11,0)}
      \draw[very thick] (origin)
      -- ++(0, 2) -- ++(1, 0)
      -- ++(0,-2) -- ++(1, 0)
      -- ++(0, 1) -- ++(1, 0)
      -- ++(0, 1) -- ++(1, 0)
      -- ++(0, 2) -- ++(1, 0)
      -- ++(0,-1) -- ++(1, 0)
      -- ++(0,-1) -- ++(1, 0)
      -- ++(0, 3) -- ++(1, 0)
      -- ++(0, 1);
      \path (origin) node(x0) {}
      ++(1,0) node(x1) {}
      ++(1,0) node(x2) {}
      ++(3,0) node(x3) {}
      ++(2,0) node(x4) {}
      ++(1,0) node(x5) {};
      \drawbrace{x0}{x1}{$W$}
      \drawbrace{x1}{x2}{$\b$}
      \drawbrace{x2}{x3}{$X$}
      \drawbrace{x3}{x4}{$Y$}
      \drawbrace{x4}{x5}{$Z$}
      \draw[very thick,dotted,blue] (origin) ++(1,0) -- ++(1,0);
      \draw[very thick,green] (origin) ++(2,1) -- ++(1,0)-- ++(0,1)-- ++(1,0)-- ++(0,2)-- ++(1,0);
    \end{tikzpicture}
  \end{center}
  \caption{Two examples of the map $\swap$ for paths with one contact: one
    where $h_X\le h_Y$ (top) and one where $h_X> h_Y$ (bottom).}\label{fig:swapI}
\end{figure}

Examples of the map $\swap$ for paths with $t(P)=1$ and $b(P)=0$ are given in Figure~\ref{fig:swapI}.
Before showing that $\swap$ is a bijection between the appropriate sets,
let us remark that, on paths with $t(P)=1$ and $b(P)=0$,
its definition is forced by the requirement that the sequence
of $y$-coordinates of east steps of $P$ that are not contacts is preserved.

\begin{prop}\label{prop:unique-path}
  There is at most one path $P_1$ in $\PP(T,B, \Set D, \Mat H)$ with $t(P_1)=1$ and $b(P_1)=0$,
  and at most one path $P_2$ with $t(P_2)=0$ and $b(P_2)=1$.
\end{prop}
\begin{proof}
  We prove the first part of the statement, the second part then follows by symmetry.
  Suppose that $P,P'\in\PP(T,B, \Set D, \Mat H)$ are two paths with $t(P)=t(P')=1$ and $b(P)=b(P')=0$.
  Since both paths have the same sequence of $y$-coordinates of non-contact east steps, we can write
  the paths as $P = U V \t W$ and $P' = U \t V W$.  It suffices to show that $V$ must be empty.

  Suppose that $V$ is not empty, and that the sequence of
  $y$-coordinates of its east steps is $v_1,v_2,\dots,v_m$.
  There cannot be a descent immediately before a top contact, in
  particular not between the last step of $V$ and $\t$ in $P$. Since the descent
  sets of $P$ and $P'$ are the same, there is no descent at this position in $P'$ either, so
  $v_{m-1}\le v_m$. But then there is no descent between the last two steps of $V$ in $P$,
  so the same is true at the corresponding position in $P'$, which implies $v_{m-2}\le v_{m-1}$.
  Repeating this argument, we obtain that $v_1\le v_2$, i.e., there is no descent between the first two steps of $V$ in $P$,
  so there is no descent between $\t$ and the first step of $V$ in $P'$. However, this is impossible: the top contact in $P'$ coincides with an east step of $T$ whose
  $y$-coordinate is strictly greater that $v_1$, since the first step of $V$ in $P$ is not a top contact.
\end{proof}

\begin{lemma}\label{lem:nocontacts}
In the decomposition $P=W X \t Y Z$ in Definition~\ref{def:swap}, neither $X$ nor $Y$ contains any top or bottom contacts.
In particular, the sequence of contacts of $\swap(P)$ is $\switch(\Mat w_P)$.
\end{lemma}

\begin{proof}
By Lemma~\ref{lem:switch}, the contact preceding the selected $\t$ must be a bottom
contact. Since $X$ contains no bottom contacts, it contains no top
contacts either.

Similarly, the contact following the selected $\t$ must be a top contact.
Since $Y$ contains no top contacts, it contains no bottom contacts either.
\end{proof}

\begin{lemma}\label{lem:swap1}
  Let $e,f,u$ be nonnegative integers with $u\ge\max\{e-f,f-e+2\}$.
  The map $\swap$ is a bijection between \begin{enumerate}
  \item the set of paths in $\PP(T, B, \Set D, \Mat H)$ whose sequence of contacts has $e$ $\t$'s, $f$ $\b$'s, and $u$ unmatched letters, and
  \item the set of paths in $\PP(T, B, \Set D, \Mat H)$ whose sequence of contacts has $e-1$ $\t$'s, $f+1$ $\b$'s, and $u$ unmatched letters.
\end{enumerate}
\end{lemma}
\begin{proof}
  Let $P$ be a path in the set described in~(i).  As in the proof of
 Lemma~\ref{lem:switch1}, $P$ must have some unmatched $\t$, so
 $\switch$ is applicable.  Moreover, the sequence of contacts of
 $\swap(P)$ is $\switch(\Mat w_P)$, which has $e-1$ $\t$'s, $f+1$
 $\b$'s, and $u$ unmatched letters.  It is clear from the
 definition that $\swap(P)\in\PP(T,B, \Mat H)$.

Let us now check that $P$ and $\swap(P)$ have the same descent set, so $\swap(P)\in\PP(T, B, \Set D, \Mat H)$.
Suppose that $P = W X \t Y Z$ is the decomposition in Definition~\ref{def:swap}.
Consider first the case $h_X\leq h_Y$, so $\swap(P) = W X Y \b Z$, that is, the block $\t Y$ in $P$ becomes $Y\b$ in $\swap(P)$.
By the choice of $Y$ and because there cannot be a descent just before a top contact, it is clear that $P$ has no descent just before
and just after the block $\t Y$, and there are descents at all positions inside the block. Let us check that this is also the case for the block $Y\b$ in $\swap(P)$.
\begin{itemize}
  \item Just before $Y\b$: if $X$ is non-empty, then $h_Y \geq h_X > -\infty$, which implies that
    $Y$ is non-empty and there is no descent between $X$ and $Y$; if $X$ is empty, then either $W$ is empty or its last
    east step has its right endpoint on $B$, so there is no descent just before $Y\b$ either.
  \item Just after $Y\b$: there cannot be a descent just after a bottom contact.
  \item Inside $Y\b$: by the definition of $Y$, there are descents at
    all positions inside $Y$; at the position between $Y$ and $\b$ (if $Y$ is non-empty), there is a descent
  because the last east step of $Y$ was not a bottom contact in $P$, by Lemma~\ref{lem:nocontacts}, so its $y$-coordinate is strictly larger than that of $\b$ in $\swap(P)$.
\end{itemize}
The arguments for the case that $h_X >  h_Y$ and $\swap(P) = W \b X Y Z$ are very similar and thus omitted.

  To show that $\swap$ is invertible we will exhibit its inverse. Informally, the description of $\swap^{-1}$ is obtained from that of $\swap$ by rotating the picture by $180$ degrees.
  Explicitly, let $P'$ be a path in the set described in~(ii).
  The proof of Lemma~\ref{lem:switch1} shows that $\Mat w_{P'}$ has some unmatched $\b$.
  The east steps of $P'$ can be decomposed uniquely as $P'=R S \b U V$, where
  \begin{itemize}
  \item the selected $\b$ is the rightmost unmatched $\b$ in $\Mat w_{P'}$,
  \item $S$ is maximal such that there is a descent after
    each of its steps, and
  \item $U$ is maximal such that there is no descent before
    any of its steps and no (left) endpoint of any of its east
    steps lies on $T$.
  \end{itemize}
  Let $h_S$ ($h_U$) be $+\infty$ if $S$ (respectively $U$) is empty,
  and otherwise the $y$-coordinate of its last (respectively first)
  east step. Define
  $$
  \swap^{-1}(P)=
  \begin{cases}
    R \t S U V & \text{if $h_S\le h_U$,} \\
    R S U \t V & \text{if $h_S> h_U$.}
  \end{cases}
  $$
  Let us check that $\swap^{-1}$ is indeed the inverse of $\swap$. Suppose that $P = W X \t Y Z$ and $h_X\leq h_Y$, and let $P'=\swap(P) = W X Y \b Z$.
  By Lemma~\ref{lem:nocontacts}, $\Mat w_{P'}=\switch(\Mat w_P)$, and so, as in the proof of Lemma~\ref{lem:switch1}, the leftmost unmatched $\t$ of $\Mat w_P$ becomes the
  rightmost unmatched $\b$ of $\Mat w_{P'}$. Additionally, when applying $\swap^{-1}$ to $P'$, the decomposition $P'=R S \b U V$
  has $S=Y$, by definition of $S$ and the fact that there is no descent just before $Y$ but there are descents in all the positions inside $Y\b$. Additionally, $h_S\le h_U$ because there was no descent just after $Y$ in $P$.
  Thus, $\swap^{-1}(P')=R \t S U V=P$. The case that $h_X>h_Y$ is similar.
\end{proof}

Figure~\ref{fig:complete-example} shows examples of the map $\swap$.

\begin{lemma}\label{lem:swapseq}
  Let $e>f$. For $0\le i\le e-f$, let $\R_i$ be the set of paths in
  $\PP(T, B, \Set D, \Mat H)$ whose sequence of contacts has $e-i$ $\t$'s,
  $f+i$ $\b$'s, and at least $e-f$ unmatched letters.
  Specifically, $\R_0$ is the set of all paths in $\PP(T, B, \Set
  D, \Mat H)$ having $e$ top contacts and $f$ bottom contacts, and
  $\R_{e-f}$ is the set of all paths in $\PP(T, B, \Set D, \Mat H)$
  having $f$ top contacts and $e$ bottom contacts (steps that are simultaneously top and a bottom contacts are disregarded there).
  Then the map $\swap$ produces a sequence of bijections
  $$\R_0\overset{\swap}\rightarrow\R_1\overset{\swap}\rightarrow\cdots\overset{\swap}\rightarrow \R_{e-f}.$$
\end{lemma}

\begin{proof}
To prove that $\R_0$ and $\R_{e-f}$ are indeed as claimed,
note that every word with $e$ $\t$'s and $f$ $\b$'s
has at least $e-f$ unmatched $\t$'s, and every word with $f$
$\t$'s and $e$ $\b$'s has at least $e-f$ unmatched $\b$'s.

By Lemma~\ref{lem:swap1}, $\swap:\R_i\rightarrow \R_{i+1}$ is a bijection for $0\le i<e-f$, since in this case $e-f\ge\max\{e-f-2i,f-e+2i+2\}$.
\end{proof}

We can now describe the bijection $\swapall$ that proves Theorem~\ref{thm:bottom-top-refined}, which in turn generalizes Theorem~\ref{thm:bottom-top}.
Figure~\ref{fig:complete-example} gives two examples of the map $\swapall$.

\begin{definition}\label{def:swapall}
For $P\in\PP(T,B)$, define $\swapall(P)=\swap^{e-f}(P)$, where $e=t(P)$  and $f=b(P)$.
\end{definition}

\begin{lemma}\label{lem:swapall}
The map $\swapall$ is an involution on $\PP(T,B)$ that preserves the descent set, as well as the sequence of
$y$-coordinates of the east steps that are not contacts, and satisfies
$t(\swapall(P))=b(P)$ and $b(\swapall(P))=t(P)$.
\end{lemma}

\begin{proof}
For any fixed $\Set D$, $\Mat H$, and $e>f$, Lemma~\ref{lem:swapseq} states that $\swap^{e-f}$ is a bijection between $\{P\in\PP(T, B, \Set D, \Mat H): t(P)=e,\ b(P)=f\}$ and $\{P\in\PP(T, B, \Set D, \Mat H): t(P)=f,\ b(P)=e\}$,
with inverse $(\swap^{-1})^{e-f}=\swap^{f-e}$. It follows that $\swapall$ is a bijection with the stated properties.
To see that it is an involution, note that for a path $P\in\PP(T,B)$ with $e$ top contacts and $f$ bottom contacts, $\swapall(\swapall(P))=\swap^{f-e}(\swap^{e-f}(P))=P$.
\end{proof}

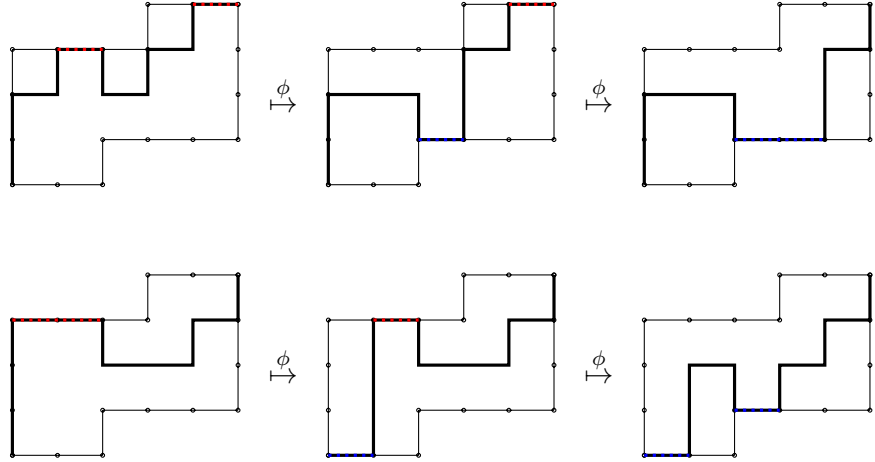
\begin{figure}[htb]
  \begin{center}
    \begin{tikzpicture}[scale=0.6]
      \def\R#1{
        \coordinate (origin) at #1;
        \draw (origin) circle(1.2pt) \E\E\N\E\E\E\N\N\N;
        \draw (origin) circle(1.2pt) \N\N\N\E\E\E\N\E\E;
      }
      \R{(0,6)}
      \draw[very thick] (origin)
      -- ++(0, 2) -- ++(1, 0)
      -- ++(0, 1) -- ++(1, 0)
      -- ++(0,-1) -- ++(1, 0)
      -- ++(0, 1) -- ++(1, 0)
      -- ++(0, 1) -- ++(1, 0);
      \draw[very thick,dotted,red] (origin) ++(1,3) -- ++(1,0);
      \draw[very thick,dotted,red] (origin) ++(4,4) -- ++(1,0);
      \draw (origin) ++(6,2) node {$\stackrel{\swap}{\mapsto}$};
      \R{(7,6)}
      \draw[very thick] (origin)
      -- ++(0, 2) -- ++(1, 0)
      -- ++(0, 0) -- ++(1, 0)
      -- ++(0,-1) -- ++(1, 0)
      -- ++(0, 2) -- ++(1, 0)
      -- ++(0, 1) -- ++(1, 0);
      \draw[very thick,dotted,blue] (origin) ++(2,1) -- ++(1,0);
      \draw[very thick,dotted,red] (origin) ++(4,4) -- ++(1,0);
      \draw (origin) ++(6,2) node {$\stackrel{\swap}{\mapsto}$};
      \R{(14,6)}
      \draw[very thick] (origin)
      -- ++(0, 2) -- ++(1, 0)
      -- ++(0, 0) -- ++(1, 0)
      -- ++(0,-1) -- ++(1, 0)
      -- ++(0, 0) -- ++(1, 0)
      -- ++(0, 2) -- ++(1, 0)
      -- ++(0, 1);
      \draw[very thick,dotted,blue] (origin) ++(2,1) -- ++(2,0);
      \R{(0,0)}
      \draw[very thick] (origin)
      -- ++(0, 3) -- ++(1, 0)
      -- ++(0, 0) -- ++(1, 0)
      -- ++(0,-1) -- ++(1, 0)
      -- ++(0, 0) -- ++(1, 0)
      -- ++(0, 1) -- ++(1, 0)
      -- ++(0, 1);
      \draw[very thick,dotted,red] (origin) ++(0,3) -- ++(2,0);
      \draw (origin) ++(6,2) node {$\stackrel{\swap}{\mapsto}$};
      \R{(7,0)}
      \draw[very thick] (origin)
      -- ++(0, 0) -- ++(1, 0)
      -- ++(0, 3) -- ++(1, 0)
      -- ++(0,-1) -- ++(1, 0)
      -- ++(0, 0) -- ++(1, 0)
      -- ++(0, 1) -- ++(1, 0)
      -- ++(0, 1);
      \draw[very thick,dotted,blue] (origin)  -- ++(1,0);
      \draw[very thick,dotted,red] (origin) ++(1,3) -- ++(1,0);
      \draw (origin) ++(6,2) node {$\stackrel{\swap}{\mapsto}$};
      \R{(14,0)}
      \draw [very thick](origin)
      -- ++(0, 0) -- ++(1, 0)
      -- ++(0, 2) -- ++(1, 0)
      -- ++(0,-1) -- ++(1, 0)
      -- ++(0, 1) -- ++(1, 0)
      -- ++(0, 1) -- ++(1, 0)
      -- ++(0, 1);
      \draw[very thick,dotted,blue] (origin)  -- ++(1,0);
      \draw[very thick,dotted,blue] (origin) ++(2,1) -- ++(1,0);
    \end{tikzpicture}
  \end{center}
  \caption{The involution $\swapall$ applied to two paths with two top contacts and no bottom contacts.}
  \label{fig:complete-example}
\end{figure}

\section{The symmetry $(b,\ell)\sim(t,r)$ for a single path}
\label{sec:bottom-left-top-right}
In this section we prove Theorem~\ref{thm:bottom-left-top-right}.
Although this theorem looks superficially similar to
Theorem~\ref{thm:bottom-top}, we have not found a comparable
\lq natural\rq\ bijective proof.  Instead, our theorem
below is a consequence of work of Anna de Mier, Joseph Bonin and Marc Noy~\cite{BoninDeMierNoy2003}, and also
Federico Ardila~\cite{Ardila2003}.

Again, let $T$ and $B$ be lattice paths in $\naturals^2$ with
north and east steps from the origin to $\F$ such that $T$ is
weakly above $B$. The paths in this section have no south steps.
We encode a path $P\in\P(T, B)$ as the subset $\nP$ of
$\Set N=\{1,2,\dots,x+y\}$ given by the indices of the north steps in
$P$.
For example, the path $P$ in Figure~\ref{fig:paths} is specified by
the subset $\nP=\{2,3,4,8,9,15,16\}\subseteq[17]$. It is shown in~\cite{BoninDeMierNoy2003}
that the set $\B=\{\nP:P\in\P(T,B)\}$ is the set of bases of a matroid with \Dfn{ground set} $\Set N$.
This matroid, which we denote by $\LPM$, is called a \Dfn{lattice path matroid}.

  Let $\prec$ be an arbitrary linear order on $\Set N$, and let
  $\nP\in\B$.  Then an element $e\not\in\nP$ is
  \Dfn{externally active} with respect to $(\nP, \prec)$ if
  $$\nexists\,n\in\nP\text{ such that } n\prec e \text{ and } \nP\setminus\{n\}\cup\{e\} \in\B,$$
  that is, the east step with index $e$ cannot be switched with a ``smaller" north step to produce another path in $\P(T,B)$.
  Similarly, an element $n\in\nP$ is \Dfn{internally active} with
  respect to $(\nP, \prec)$ if
  $$\nexists\,e\in \Set N\setminus \nP\text{ such that } e\prec n \text{ and } \nP\setminus\{n\}\cup\{e\}\in\B,$$
  that is, the north step with index $n$ cannot be switched with a ``smaller" east step to produce another path in $\P(T,B)$.

  For a fixed order $\prec$, the \Dfn{internal activity} of $\nP$ is the number of
  internally active elements with respect to $(\nP,\prec)$, and similarly its \Dfn{external activity} is the number
  of externally active elements.  The \Dfn{Tutte polynomial} of $\LPM$, introduced by Henry
  Crapo~\cite{Crapo1969} generalizing Tutte's dichromate for graphs, is the generating polynomial for the internal and external
  activities of its bases:
  \begin{equation}\label{eq:activity}
    \sum_{\nP\in\B} x^\text{internal activity of $\nP$}
                      y^\text{external activity of $\nP$}.
  \end{equation}

The main ingredient in the proof of
Theorem~\ref{thm:bottom-left-top-right}, restated below for convenience, is the fact that the Tutte polynomial is well-defined, i.e., independent of the ordering of the ground set.  In
Appendix~\ref{sec:Tutte-independence} we give an activity-preserving
bijection on the bases of a matroid relative to two orderings that differ only in the order of two covering elements.
This bijection works for any matroid, and in particular for $\LPM$, so it gives a bijective proof of Theorem~\ref{thm:bottom-left-top-right-restated},
but it is iterative and rather tedious to apply. It would be interesting to find a direct bijective proof of this theorem.

\begin{theorem}\label{thm:bottom-left-top-right-restated}
  The pairs $(b, \ell)$ and $(t, r)$ have the same joint distribution over $\P(T, B)$.
\end{theorem}
\begin{proof}
  Let $\prec$ be the usual order $1\prec 2\prec 3\prec\cdots$ of the
  ground set $\Set N$ of the matroid $\LPM$ described above. As shown in~\cite[Theorem 5.4]{BoninDeMierNoy2003}, an internally active element of
  $\nP\in\B$ with respect to this order is a left
  contact of $P\in\P(T,B)$, and an externally active element of $\nP$ is a bottom contact of $P$. By equation~\eqref{eq:activity}, the Tutte polynomial of $\LPM$ equals $\sum_{P\in\P(T,B)} x^{b(P)}y^{\ell(P)}$.

  Since the Tutte polynomial is independent of the ordering on the
  ground set, the same polynomial can also be obtained as follows.
  Let now $\prec$ be the order $\cdots\prec 3\prec 2\prec 1$. With respect to this order,
  we claim that an internally active element of $\nP\in\B$ is a right contact of $P\in\P(T,B)$, and an externally
  active element of $\nP$ is a top contact of $P$. From the claim it follows that the Tutte polynomial of $\LPM$ also equals $\sum_{P\in\P(T,B)} x^{t(P)}y^{r(P)}$, and thus
  $$\sum_{P\in\P(T,B)} x^{b(P)}y^{\ell(P)}=\sum_{P\in\P(T,B)} x^{t(P)}y^{r(P)}$$ as desired.

  It remains to prove the claim, which can be done with the following argument analogous to~\cite[Theorem 5.4]{BoninDeMierNoy2003}.
  Let $P\in\P(T, B)$, and suppose that $n\in \nP$, that is, the $n$-th step of $P$ is a north step.
  If this step is a right contact, then for any $e$ satisfying that
  $e\notin\nP$ (i.e., the $e$-th step of $P$ is an east step) and $e\prec n$ (i.e., $e>n$), the path whose north steps are $\nP\setminus\{n\}\cup\{e\}$ does not lie
  weakly above $B$, since the $n$-th step of this path goes under $B$, so $\nP\setminus\{n\}\cup\{e\}\notin\B$. Thus $n$ is internally active.

  Conversely, if the $n$-th step of $P$ is not a right contact, let
  $e$ be the index of the first east step with $e>n$ that touches $B$ with its right
  endpoint.  Then $e\notin\nP$, $e\prec n$, and the path whose north steps are $\nP\setminus\{n\}\cup\{e\}$ belongs to $\P(T,B)$, so $\nP\setminus\{n\}\cup\{e\}\in\B$. Thus $n$ is not internally active.

  The argument for externally active edges is very similar and thus omitted.
\end{proof}

\section{A $k$-tuple of paths between two boundaries}
\label{sec:path-tuple}

In this section we show how to extend Theorems~\ref{thm:bottom-top} and~\ref{thm:bottom-left-top-right} to families of $k$
non-crossing paths.  In both cases, the idea is to
repeatedly apply the theorems for single paths.

\subsection{The symmetry $(t,b)\sim(b,t)$}
\label{sec:bottom-top-k}

In our extension of Theorem~\ref{thm:bottom-top} to $k$-tuples of paths, we do not allow paths with south steps, unlike in the more general Theorem~\ref{thm:bottom-top-refined}.
The reason is that the distribution of bottom and top contacts over $\P(T, B)$ is not symmetric if we allow south steps in the boundary paths $T$ and $B$.
However, we are able to partially incorporate the refinement keeping track of the $y$-coordinates of the non-contact east steps.

Recall from Section~\ref{sec:statement} that, given
$\Mat P=(P_1,P_2,\dots,P_k)\in\P^k(T, B)$, with the convention that $P_0 = T$
and $P_{k+1} = B$, we denote by $h_i=h_i(\Mat P)$ the number of east steps where $P_i$ and $P_{i+1}$ coincide, for $0\le i\le k$.
Recall that $T$ and $B$ are paths from $(0,0)$ to $(x,y)$.
For $1\le s\le y-1$, let $u_s=u_s(\Mat P)$ be the number of east steps with $y$-coordinate $y-s$
that lie strictly between $T$ and $B$ and are not used by any of the paths $P_1,\dots,P_{k}$.
Let $\Mat u(\Mat P)=(u_1,u_2,\dots)$.
For example, the $k$-tuple in Figure~\ref{fig:k-tuple} has $(h_0,h_1,h_2)=(4,4,6)$ and $(u_1,\dots,u_6)=(3,0,1,2,3,2)$.
For a fixed sequence $\Mat u$ of nonnegative integers,
let $\P^k(T,B, \Mat u)$ be the set of $k$-tuples $\Mat P\in\P^k(T, B)$ with $\Mat u(\Mat P)=\Mat u$.

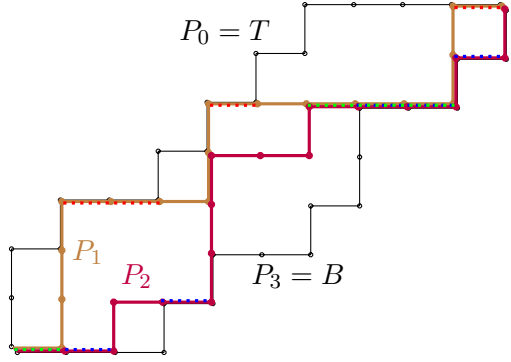
\begin{figure}[htb]
  \begin{center}
\begin{tikzpicture}[scale=0.65]
 \draw (-0.06,0.06) circle(1.2pt) \N\N\E\N\E\E\N\E\N\E\N\E\N\E\E\E\E;
 \draw (0.06,-0.06) circle(1.2pt) \E\E\E\N\E\N\E\E\N\E\N\N\E\E\N\E\N;
 \draw[very thick,brown] (-0.03,0.03) \e\n\n\n\e\e\e\n\n\e\e\e\e\e\n\n\e;
 \draw[very thick,purple] (0.03,-0.03) \e\e\n\e\e\n\n\n\e\e\n\e\e\e\n\e\n;
 \draw[very thick,dotted,red] (1,3) -- (3,3);
 \draw[very thick,dotted,red] (4,5) -- (5,5);
 \draw[very thick,dotted,red] (9,7) -- (10,7);
 \draw[very thick,dotted,blue] (0,0) -- (2,0);
 \draw[very thick,dotted,blue] (3,1) -- (4,1);
 \draw[very thick,dotted,blue] (7,5) -- (9,5);
 \draw[very thick,dotted,blue] (9,6) -- (10,6);
 \draw[very thick,dotted,green] (0,0) -- (1,0);
 \draw[very thick,dotted,green] (6,5) -- (9,5);
 \draw[brown] (1.5,2) node {$P_1$};
 \draw[purple] (2.5,1.5) node {$P_2$};
 \draw (4.3,6.5) node {$P_0=T$};
 \draw (5.8,1.5) node {$P_3=B$};
\end{tikzpicture}
  \end{center}
  \caption{A tuple of paths in $\P^2(T, B)$.}
  \label{fig:k-tuple}
\end{figure}

\begin{theorem}\label{thm:bottom-top-k-refined}
  For any sequence $\Mat u$ of nonnegative integers, the distribution of $(h_0,h_1,\dots,h_k)$ over $\P^k(T,B, \Mat u)$ is symmetric.
\end{theorem}
\begin{remark}
  Disregarding the refinement by $\Mat u$, two recursive, non-bijective proofs of this theorem have been given by
  Nicolas~\cite[Theorem~3]{Nicolas2009} and Han~\cite{Han}. To the extent of our knowledge, ours is the first bijective proof.
\end{remark}
\begin{proof}
  It suffices to show that
  $(h_0,\dots,h_{i-1},h_{i},\dots,h_k)\sim(h_0,\dots,h_{i},h_{i-1},\dots,h_k)$ over $\P^k(T,B, \Mat u)$ for any $i$ with $1\le i\le k$.
  Fix such an $i$, and let $\Mat P=(P_1,P_2,\dots,P_k)\in\P^k(T,B, \Mat u)$. Regarding $P_i$ as
  a path in $\P(P_{i-1},P_{i+1})$, we can apply the bijection $\swapall$ from
  Definition~\ref{def:swapall} to it, obtaining a path $\swapall(P_i)\in \P(P_{i-1},P_{i+1})$. Let $\Mat Q=(Q_1,Q_2,\dots,Q_k)$, where $Q_i = \swapall(P_i)$ and
  $Q_j=P_j$ for $j\neq i$.

  By definition, $h_j(\Mat P)=h_j(\Mat Q)$ for $j\notin\{i-1,i\}$. By Lemma~\ref{lem:swapall}, the number of east steps where  $P_i$ and $P_{i+1}$ (respectively $P_{i-1}$) coincide equals the number of east steps where $\swapall(P_i)$ and $P_{i-1}$ (respectively $P_{i+1}$) coincide,
  so $h_i(\Mat P)=h_{i-1}(\Mat Q)$ and $h_{i-1}(\Mat P)=h_{i}(\Mat Q)$.

  It remains to show that $\Mat u(\Mat P)=\Mat u(\Mat Q)$. The only east steps that need to be checked are those that lie strictly between $P_{i-1}=Q_{i-1}$ and $P_{i+1}=Q_{i+1}$.
  We know by Lemma~\ref{lem:swapall} applied to $P_i\in\P(P_{i-1},P_{i+1})$ that the
  multiset of $y$-coordinates of east steps of $P_i$ that do not coincide with either $P_{i-1}$ or $P_{i+1}$ equals the corresponding multiset of $Q_i$. Thus, the number of unused east steps
  lying between $P_{i-1}$ and $P_{i+1}$ at each fixed $y$-coordinate is the same in both $\Mat P$ and $\Mat Q$.
\end{proof}

\subsection{The symmetry $(b,\ell)\sim(t,r)$}
\label{sec:bottom-left-top-right-k}

Given $\Mat P=(P_1,P_2,\dots,P_k)\in\P^k(T, B)$, letting $P_0 = T$ and $P_{k+1} = B$
we denote by $v_i=v_i(\Mat P)$ the number of north steps where $P_i$ and $P_{i+1}$ coincide
and, as before, by $h_i=h_i(\Mat P)$ the number of east steps where $P_i$ and $P_{i+1}$ coincide, for $0\le i\le k$.
Note that $t(\Mat P) = h_0(\Mat P)$, $b(\Mat P) = h_k(\Mat P)$,
$\ell(\Mat P) = v_0(\Mat P)$ and $r(\Mat P) = v_k(\Mat P)$.

\begin{theorem}\label{thm:bottom-left-top-right-k-restated}
  The pairs $(b, \ell)$ and $(t, r)$ have the same joint distribution over
  $\P^k(T, B)$.
\end{theorem}
\begin{proof}
  As in the proof of Theorem~\ref{thm:bottom-top-k-refined},
  each $P_i$, for $1\le i\le k$, can be regarded as a path in $\P(P_{i-1},P_{i+1})$.
  In this setting, the statistics involved in the statement of Theorem~\ref{thm:bottom-left-top-right-restated} are
  $b(P_i)=h_{i}(\Mat P)$, $\ell(P_i)=v_{i-1}(\Mat P)$, $t(P_i)=h_{i-1}(\Mat P)$ and $r(P_i)=v_{i}(\Mat P)$.
  Applying Theorem~\ref{thm:bottom-left-top-right-restated} to $P_i\in\P(P_{i-1},P_{i+1})$,
  it follows that there is a bijection between tuples $\Mat P\in\P^k(T, B)$ with $h_{i}(\Mat P) = e$ and $v_{i-1}(\Mat P) = f$,
  and tuples $\Mat P\in\P^k(T, B)$ with $h_{i-1}(\Mat P) = e$ and $v_{i}(\Mat P) = f$, which preserves the statistics $h_j$ and $v_j$ for all $j\notin\{i-1,i\}$.

  Given $\Mat P\in\P^k(T, B)$ with $b(\Mat P)=h_k(\Mat P) = e$ and $\ell(\Mat P)=v_0(\Mat P) = f$, one can apply
  this bijection first to $P_k$, then to $P_{k-1}$ in the resulting tuple, and successively up to $P_1$. This composition gives a bijection between
  tuples $\Mat P\in\P^k(T, B)$ with $b(\Mat P)= e$ and $\ell(\Mat P)= f$, and tuples $\Mat P\in\P^k(T, B)$
  with $t(\Mat P)=h_0(\Mat P) = e$ and $v_{1}(\Mat P) = f$.
  On the latter set, one can now apply the bijection to $P_{2}$, then to $P_{3}$, and successively
  down to $P_k$, proving that tuples $\Mat P\in\P^k(T, B)$ with $t(\Mat P)= e$ and $v_{1}(\Mat P) = f$ are in turn in bijection with
  tuples $\Mat P\in\P^k(T, B)$ with $t(\Mat P)=e$ and $r(\Mat P)=v_k(\Mat P)=f$.
\end{proof}

\section{Corollaries and Applications}\label{sec:applications}
\subsection{Dyck paths and generalizations}\label{sec:Dyck}
In the particular case that $T=N^n E^n$ and $B=(EN)^n$, the
statistics $t$ and $b$ become two familiar Dyck path statistics: the
height of the last peak and the number of returns, respectively.  A
bijective proof of the fact that these statistics are equidistributed
on Dyck paths was given by Deutsch~\cite{Deutsch1998}, who later also
exhibited a recursively defined involution~\cite{Deutsch1999} proving the symmetry of their
joint distribution. The equidistribution result is also a consequence of the more recent bijection $\zeta$ due to Haglund \cite[p. 50]{MR2371044} and Andrews {\it et al.}~\cite{MR1926854}.

Deutsch's involution has recently been rediscovered by Bousquet-M\'elou, Fusy and Pr\'eville-Ratelle~\cite{BousquetMelou}, who consider the bijection between binary trees and Dyck paths obtained by reading the tree in postorder and recording an $N$ for each leaf (except the first one) and an $E$ for each internal node. As they point out, the involution on binary trees produced by reflecting along a vertical axes translates via this bijection into an involution on Dyck paths that switches the statistics $t$ and $b$. It is not hard to show that this operation coincides with Deutsch's involution (up to a minor modification in order to deal with the height of the last peak rather than the first), and in fact it provides a non-recursive description of it.

The symmetry of the statistics `height of the last peak' and `number of returns' on Dyck paths
can also be proved using standard generating function techniques, based on the usual recursive decomposition of Dyck paths.
However, neither these techniques nor the above bijections seem to
extend to the general setting of Theorem~\ref{thm:bottom-top}.
Our involution $\swapall$, when restricted to the
case of Dyck paths, is quite different from Deutsch's involution
and Haglund's bijection.  In addition to providing an extension to paths between arbitrary
boundaries $T$ and $B$, our involution can also be used to prove the following.

\begin{cor}\label{cor:i+j}
Let $T$ and $B$ be arbitrary paths with $N$ and $E$ steps from
$(0,0)$ to $(x,y)$ and $T$ weakly above $B$.
The following statements are equivalent:
  \begin{enumerate}
  \item the number of paths in $\P(T,B)$ with $i$ top and $j$ bottom
    contacts depends only on $i+j$;
  \item for every path in $\P(T,B)$, all its
    bottom contacts occur before its top contacts;
  \item the last east step of $B$ is lower than the first east step
    of $T$.
  \end{enumerate}
Similarly, the following statements are equivalent:
  \begin{enumerate} \renewcommand{\labelenumi}{(\roman{enumi}')}
  \item if $\P(T,B,\Mat H)\neq\emptyset$ and $i+j+\size{\Mat H}=x$, then $\P(T,B,\Mat H)$ contains precisely one path with $i$ top and $j$ bottom contacts;
  \item for every path in $\P(T,B,\Mat H)$, all its bottom contacts occur before its top contacts.
  \end{enumerate}
\end{cor}

\begin{proof}
  Let us first show that (ii') implies (i').  When the bottom
  contacts of a path occur before its top contacts, all the letters
  in its sequence of contacts are unmatched.  By assumption, this is
  the case for all paths in $\P(T, B, \Mat H)$.  Thus we can apply
  Lemma~\ref{lem:swapseq}: let $\Set D=\emptyset$, $f=0$ and
  $e=x-\size{\Mat H}$, the number of contacts of a path in $\P(T, B,
  \Mat H)$.  We then obtain a sequence of bijections
  $\R_0\overset{\swap}\rightarrow
  \R_1\overset{\swap}\rightarrow\cdots\overset{\swap}\rightarrow
  \R_e$, where
  \[
  \R_s=\{P\in\P(T, B, \Mat H):t(P)=e-s,\ b(P)=s\}.
  \]
  Finally, as in the proof of Proposition~\ref{prop:unique-path} we
  can see that there is at most one path in $\R_0$, because this is
  just the set of paths in $\P(T, B, \Mat H)$ with no bottom
  contacts.

  Conversely assume (i') and $\P(T,B,\Mat H)\neq\emptyset$.  Then
  there is precisely one path in $\R_0$.  The sequence of contacts in
  this path contains only unmatched letters.  Since $\swap$ preserves
  the number of unmatched letters, all paths in $\P(T,B,\Mat H)$ have
  sequences of contacts with unmatched letters only.  This implies
  that a bottom contact cannot appear after a top contact in any
  path.

  The equivalence of (i') and (ii') entails the equivalence of (i)
  and (ii), because we can use (i') and (ii') for all choices of
  $\Mat H$ separately.  Condition (ii) implies (iii) because if the
  last east step of $B$ is at the same height as the first east step
  of $T$ or higher we can choose a path beginning with a top contact
  and ending with a bottom contact.  Finally, (iii) obviously implies
  (ii).
\end{proof}

We remark that the naive generalization of Corollary~\ref{cor:i+j}
to $k$-tuples of paths does not hold.  For example, in the set $\P^2(T,B)$ where $T=NNEE$ and $B=ENEN$, there is only one pair of paths with
$(h_0,h_1,h_2)=(0,1,2)$, but there are two pairs of paths with
$(h_0,h_1,h_2)=(1,1,1)$.

\begin{cor}\label{cor:easy-bottom}
  Let $T=N^yE^x$, and let $B$ be any path from the origin to $(x,y)$, weakly below $T$
  and ending with a north step. Then the number of
  paths in $\P(T, B)$ with $i$ top and $j$ bottom contacts equals the number of paths with north and east
  steps from the origin to $(x-i-j, y-2)$ staying weakly above $B$.
\end{cor}

\begin{proof}
  By Corollary~\ref{cor:i+j}, it is enough to count paths in $\P(T, B)$ with $c:=i+j$ top contacts and no bottom contacts.
  Such paths are in bijection (by removing the terminal $NE^c$) with paths from the origin to $(x-c, y-1)$
  strictly above $B$, which in turn are in bijection (by removing the initial $N$) with paths from
  the origin to $(x-c, y-2)$ weakly above $B$, as claimed.
\end{proof}

In particular, Corollary~\ref{cor:easy-bottom} allows us to refine a formula due to
Vladimir Korolyuk~\cite{MR0067418} for the number of lattice paths
above a linear boundary whose slope is an integer.  For natural
numbers $m$, $n$, $r$ and $k$, let
$p_{r,k}(m,n)$ be the number of lattice paths with north and east steps from the origin to $(m,n)$ that do not go below the line $y = rx-k$.
Korolyuk's formula is equivalent to
\begin{equation}\label{eq:Korolyuk}
  p_{r,k}(m,n) = \sum_{i=0}^{\lfloor\frac{k}{r+1}\rfloor}%
  (-1)^i\frac{n+k+1-rm}{n+k+1-ri}\binom{m+n+k-(r+1)i}{m-i}\binom{k-ri}{i}
\end{equation}
for $n\ge rm-k$. When $k=0$, this equation trivially reduces to the generalized ballot theorem,
\begin{equation}\label{eq:ballot}
  p_{r,0}(m,n) = \frac{n+1-rm}{n+1}\binom{m+n}{m} %
  = \binom{m+n}{m} - r \binom{m+n}{m-1}.
\end{equation}
Amazingly, for $r=1$ Equation~\eqref{eq:Korolyuk} also has a closed form,
\begin{equation*}
  p_{1,k}(m,n) %
  = \binom{m+n}{m} - \binom{m+n}{m-1-k},
\end{equation*}
which is not hard to prove using the reflection principle.  For some
information on the history of these results, see Marc Renault's
note~\cite{MR2398417}.

Using Korolyuk's formula, Corollary~\ref{cor:easy-bottom} gives a formula for the number of paths from $(0,0)$ to $(m,n)$ not going below the line $y = rx-k$ and having $i$ top and $j$ bottom contacts, assuming that $n > rm-k$.
By reflecting the picture about the line $y=-x$, a similar result can be obtained for paths not going below a line of slope $1/r$.

It may be surprising that these are the only known formulas that we can use in applications of
Corollary~\ref{cor:easy-bottom}.  Indeed,
all the explicit formulas or generating function solutions for the
number of paths above other boundaries we are aware of, e.g.
\cite{MR1909568,MR0061567,MR2469258,MR1798327,MR570213,MR1456723,MR2500155},
are such that the starting and endpoints must lie on the boundary and
thus cannot be used in conjunction with
Corollary~\ref{cor:easy-bottom}.

\subsection{A decomposition of the generalized Tamari lattice into
  chains}
The generalized Tamari lattices were introduced by Bergeron in the
context of diagonal harmonics~\cite{Bergeron2011}.  In this section,
we exhibit a connection between the covering relation in these
lattices and the map $\swap$ introduced in Definition~\ref{def:swap}.

For an integer $r\geq 1$, the elements of the $n$-th $r$-Tamari
lattice are the paths in $\P\big(T_n^{(r)}, B_n^{(r)}\big)$ where
$T_n^{(r)} = N^n E^{rn}$ and $B_n^{(r)} = \left(NE^r\right)^n$.  By
Equation~\eqref{eq:ballot}, the generalized Tamari lattice has
$p_{r,0}(n,rn) = \frac{1}{rn+1}\binom{(r+1)n}{n}$ elements.

Let $Q$ be a path in $\P\big(T_n^{(r)}, B_n^{(r)}\big)$ regarded as a
sequence of east steps, disregarding the final $r$ east steps, following our convention to ignore steps which are simultaneously top and bottom contacts.
Consider any decomposition of $Q$ of the
form $Q = W \h_1 X \h_2 Z$, where $\h_1$ and $\h_2$ are single east
steps with different $y$-coordinates (possibly top or bottom contacts), 
$X$ is maximal (possibly empty) such that the right endpoints of
all its steps are strictly above the line of slope $1/r$ beginning at
the right endpoint of $\h_1$.  Then $Q$ is covered by the path $P = W
X \h_2 \h_3 Z$, where $\h_3$ is a single east step with the same
$y$-coordinate as $\h_2$.  The partial order in the generalized
Tamari lattice, which we denote by $\preceq$, is the reflexive and
transitive closure of this covering relation.  Its top element is
$T_n^{(r)}$ and its bottom element is~$B_n^{(r)}$.

The connection to the map $\swap$ from Definition~\ref{def:swap} is
now immediate:
\begin{prop}
  Let $P\in\P\big(T_n^{(r)}, B_n^{(r)}\big)$ such that $\swap(P)$ is defined.  
  Then $P$ covers $\swap(P)$ in the generalized Tamari lattice.
\end{prop}
\begin{proof}
  We can write the east steps of $P$ uniquely as $P=W X \t\t \dots
  \t$, where $X$ is maximal not containing top or bottom
  contacts.  Then $\swap(P)=W \b X \t \dots \t$.  Since all steps in
  $X$ are above the line of slope $1/r$ beginning at the right
  endpoint of $\b$ in $\swap(P)$, $P$ indeed covers $\swap(P)$.
\end{proof}

\begin{cor}
  Let $\R_0$ be the set of paths in $\P\big(T_n^{(r)},
  B_n^{(r)}\big)$ without proper bottom contacts, i.e., which are not
    simultaneously top contacts. Then
  \[
  \bigcup_{P\in\R_0} \big(P, \swap(P),\swap^2(P),\dots,\swapall(P)\big)
  \]
  is a decomposition of the generalized Tamari lattice into
  $p_{r,0}\big(n-1,r(n-1)\big) =
  \frac{1}{r(n-1)+1}\binom{(r+1)(n-1)}{n-1}$ saturated chains.
\end{cor}
It turns out that, for $r=1$, this decomposition is symmetric in the
following sense.
\begin{prop}
  For $P,Q\in \P\big(T_n^{(r)}, B_n^{(r)}\big)$ with $Q\preceq P$,
  let $d(Q,P)$ be the minimal distance from $Q$ to $P$ in the Hasse
  diagram of the generalized Tamari lattice.  Then $d\big(P,
  T_n^{(r)}\big) = rn-t(P)$.  For the classical Tamari lattice we
  have $d\big(B_n^{(1)}, P\big) = n-b(P)$.
\end{prop}
\begin{proof}
  Let us first show that for any path $Q$ we have $d\big(Q,
  T_n^{(r)}\big)\geq rn-t(Q)$.  By definition of the covering
  relation, the covering path can have at most one more top contact
  than the path covered.  Applying this to a saturated chain from $Q$
  to $T_n^{(r)}$ of length $d\big(Q, T_n^{(r)}\big)$, we obtain that
  $t(Q) + d\big(Q, T_n^{(r)}\big) \geq t\big(T_n^{(r)}\big) = rn$.

  To show that equality holds, we use induction on $D(Q) = rn - t(Q)$.  We have $D(Q) = 0$ only
  for $Q=T_n^{(r)}$, in which case the formula for the distance is
  trivially correct.  Otherwise, decompose $Q$ as $Q = W \h_1 \t
  \dots \t$ where $\h_1$ is the last east step of $Q$ which is not a
  top contact.  Thus $Q$ is covered by the path $P = W \t\t \dots \t$
  and $t(P)=t(Q)+1$.  By the induction hypothesis the formula is
  correct for $P$.  Since $d\big(Q, T_n^{(r)}\big)\leq d\big(P,
  T_n^{(r)}\big) + 1$, we have
  \[
  rn-t(Q) \leq d\big(Q, T_n^{(r)}\big) \leq d\big(P, T_n^{(r)}\big) +
  1 = rn - t(P) + 1,
  \]
  and therefore equality.

  We prove the formula for the distance of a path to the bottom
  element of the classical Tamari lattice (i.e., $r=1$) similarly,
  with the two main ingredients being the following.  First, we use
  the fact that in any covering relation, the covered path can have at most one more bottom contact
  than the covering path.  Let us remark here
  that the generalized Tamari lattice does not have this property.
  Second, if $P\neq B_n^{(1)}$, we decompose it as $P = W X \h_2 \h_3
  Z$ where $\h_3$ is the last bottom contact preceded by an east step
  with the same $y$-coordinate, and $X$ is maximal such that the
  right endpoints of all its (east) steps are strictly above the line
  of slope $1$ passing through the right endpoint of $\h_2$.  Then
  $P$ covers $Q = W \h_1 X \h_2 Z$, where $\h_1$ and $\h_2$ are both
  bottom contacts.  In both $P$ and $Q$ the east steps in $X$ do not
  contain any bottom contacts.
\end{proof}

\subsection{Two conjectures}

Given Theorems~\ref{thm:bottom-top}
and~\ref{thm:bottom-left-top-right}, it is natural to ask whether
there are other pairs of lattice paths statistics involving contacts
that have the same distribution.  As a first tentative answer to this
question we offer the following conjecture:
\begin{cnj}
  Suppose that $T$ and $B$ touch only at the origin and at
  their common endpoint.  Then the following statements are
  equivalent:
  \begin{enumerate}
  \item the pairs $(b, \ell)$ and $(b, t)$ have the same joint
    distribution over $\P(T, B)$,
  \item the pairs $(b, \ell)$ and $(\ell, r)$ have the same joint
    distribution over $\P(T, B)$,
  \item the pairs $(t, r)$ and $(b, t)$ have the same joint
    distribution over $\P(T, B)$,
  \item the pairs $(t, r)$ and $(\ell, r)$ have the same joint
    distribution over $\P(T, B)$,
  \item either $T = (N E)^n$ or $B = (E N)^n$ for some $n$.
  \end{enumerate}
\end{cnj}
Rotating the region by $180$ degrees, or reflecting it about the line
$y=-x$ or $y=x$ we see that it is sufficient to prove the equivalence
of the joint distribution of $(b, \ell)$ and $(b, t)$ with the
assertion that either $T = (N E)^n$ or $B = (E N)^n$.  

Moreover, when $T = (N E)^n$, each top contact coincides with a left
contact.  Thus, in this case it is clear that $(b, \ell)$ and $(b,
t)$ have the same joint distribution.  When $B = (E N)^n$, each
bottom contact coincides with a right contact and we obtain
$(b,\ell)\sim(t,r)\sim(t,b)\sim(b,t)$ by applying
Theorems~\ref{thm:bottom-top} and~\ref{thm:bottom-left-top-right}.

We remark that $(b,r)\sim(r,b)$ seems to force symmetry of the bottom
boundary $B$, but symmetry of the top boundary $T$ is neither
sufficient nor necessary, as the example $T=N^2 E N^2 E^3$ shows.

Another natural question arises from Corollary~\ref{cor:i+j}, which
deals only with the statistics $(b, t)$: are there regions for which
we have an analogous result for the pair $(b, \ell)$?  It appears
that the answer is affirmative only when the result holds trivially:
\begin{cnj}
  Suppose that $T$ and $B$ touch only at the origin and at
  their common endpoint.  Then the number of paths in $\P(T, B)$ with
  $i$ bottom and $j$ left contacts depends only on $i+j$ if and only
  if $T = N^n E^n$ and $B = (E N)^n$, or $T = (N E)^n$ and $B = E^n
  N^n$ for some $n$.
\end{cnj}

We checked the above two conjectures for all pairs of paths $T$ and $B$ from the origin to
$(n,n)$ for $n\leq 6$.

\subsection{Patterns in permutations}

We now describe an application of
Theorem~\ref{thm:bottom-top-refined} to restricted
permutations. Let $\S_n$ denote the set of permutations of $[n]$.

\begin{definition}
  Let $\pi\in \S_n$. We say that $\pi(i)$ is a \Dfn{right-to-left
    minimum} (\Dfn{right-to-left maximum}) of $\pi$ if $\pi(i) <
  \pi(j)$ (respectively, $\pi(i)>\pi(j)$) for all $j>i$.
  For $1<i<n$, we say that $\pi$ has an \Dfn{occurrence of the (dashed) pattern $13\mn2$} at position $i$ if there is a $j>i+1$ such that $\pi(i) <
  \pi(j) < \pi(i+1)$.
\end{definition}

For example, the permutation $35681742$ has occurrences of $13\mn2$ at positions $1$, $3$, and $5$. The right-to-left minima of this permutation are $1,2$, and its right-to-left maxima are $8,7,4,2$.

\begin{prop}\label{prop:permutations}
  The set of permutations in $\S_n$ with $e$ right-to-left minima, $f$ right-to-left maxima, and having occurrences of the pattern $13\mn2$ exactly at
  positions $\Set D$ is in bijection with the set of
  paths in $\PP(T, B, \Set D)$ with $e$ top contacts and $f$ bottom contacts, where $T=N^nE^n$ and $B=(EN)^n$.
\end{prop}
\begin{proof}
  Given a path $P\in\PP(T, B, \Set D)$, let $y_1,y_2,\dots,y_n$ be the sequence of $y$-coordinates of its east steps from left to right, and let
  $p_i=n+1-y_i$ for $1\le i\le n$.
  We associate to $P$ a permutation $\pi\in\S_n$ as follows: $\pi(1)=p_1$ and, for each $i$ from $2$ to $n$, let $\pi(i)$ be the $p_i$-th smallest
  number in $[n]\setminus\{\pi(1),\pi(2),\dots,\pi(i-1)\}$. An example is drawn in Figure~\ref{fig:perm2path}.

  With this definition, the $i$-th east step of $P$ is a top contact if $p_i=1$, which happens if and only if $\pi(i)$ is the smallest
  number in $\{\pi(i),\pi(i+1),\dots,\pi(n)\}$, and hence a right-to-left minimum.  Similarly,
  the $i$-th east step of $P$ is a bottom contact if and only if $\pi(i)$ is a right-to-left maximum.

  Next we prove that descents in $P$ correspond to occurrences of the pattern
  $13\mn2$ in $\pi$. Fix $i$, and suppose that the elements in $[n]\setminus\{\pi(1),\pi(2),\dots,\pi(i-1)\}$, when listed in increasing order,
  are $r_1<r_2<\dots<r_{n-i}$. Let $a=p_i$, $b=p_{i+1}$, and note that $\pi(i)=r_a$ by definition.
  Clearly, $P$ has a descent in position $i$ if and only if $y_i>y_{i+1}$, which is equivalent to $a<b$.
  In this case, $\pi(i+1)=r_{b+1}$, and there is $j>i$ such that $\pi(j)=r_b$, so that $\pi(i)\pi(i+1)\pi(j)$ is an occurrence of $13\mn2$.
  On the other hand, if $P$ has no descent in position $i$, there are two possibilities.
  If $a>b$, then $\pi(i+1)=r_b < r_a =\pi(i)$; if $a=b$, then $\pi(i+1)=r_{a+1}>r_a=\pi(i)$, but there is no $j>i$ with
  $\pi(i)<\pi(j)<\pi(i+1)$. In both cases, $\pi$ has no occurrence of $13\mn2$ at position $i$.
\end{proof}

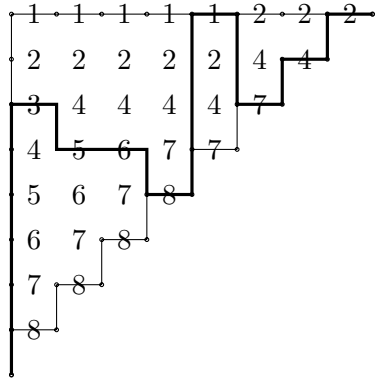
\begin{figure}[htb]
  \begin{center}
    \begin{tikzpicture}[scale=0.6]
        \coordinate (origin) at (0,0);
        \draw (origin) circle(1.2pt) \N\E\N\E\N\E\N\E\N\E\N\E\N\E\N\E;
        \draw (origin) circle(1.2pt) \N\N\N\N\N\N\N\N\E\E\E\E\E\E\E\E;
        \draw[very thick] (origin)
        -- ++(0,6) -- ++(1,0) -- ++(0,-1) -- ++(2,0) -- ++(0,-1) -- ++(1,0) -- ++(0,4) -- ++(1,0)  -- ++(0,-2) -- ++(1,0) -- ++(0,1) -- ++(1,0) -- ++(0,1) -- ++(1,0);
        \draw (0.5,1) node {$8$};
        \draw (0.5,2) node {$7$};
        \draw (0.5,3) node {$6$};
        \draw (0.5,4) node {$5$};
        \draw (0.5,5) node {$4$};
        \draw (0.5,6) node {$3$};
        \draw (0.5,7) node {$2$};
        \draw (0.5,8) node {$1$};
        \draw (1.5,2) node {$8$};
        \draw (1.5,3) node {$7$};
        \draw (1.5,4) node {$6$};
        \draw (1.5,5) node {$5$};
        \draw (1.5,6) node {$4$};
        \draw (1.5,7) node {$2$};
        \draw (1.5,8) node {$1$};
        \draw (2.5,3) node {$8$};
        \draw (2.5,4) node {$7$};
        \draw (2.5,5) node {$6$};
        \draw (2.5,6) node {$4$};
        \draw (2.5,7) node {$2$};
        \draw (2.5,8) node {$1$};
        \draw (3.5,4) node {$8$};
        \draw (3.5,5) node {$7$};
        \draw (3.5,6) node {$4$};
        \draw (3.5,7) node {$2$};
        \draw (3.5,8) node {$1$};
        \draw (4.5,5) node {$7$};
        \draw (4.5,6) node {$4$};
        \draw (4.5,7) node {$2$};
        \draw (4.5,8) node {$1$};
        \draw (5.5,6) node {$7$};
        \draw (5.5,7) node {$4$};
        \draw (5.5,8) node {$2$};
        \draw (6.5,7) node {$4$};
        \draw (6.5,8) node {$2$};
        \draw (7.5,8) node {$2$};
    \end{tikzpicture}
  \end{center}
  \caption{The path corresponding to the permutation $35681742$.} \label{fig:perm2path}
\end{figure}

\begin{cor}
Let $\Set D\subseteq[n-1]$. In the set of permutations $\pi\in\S_n$ having occurrences of $1\mn32$ exactly at positions $\Set D$, the joint distribution of the statistics
`number of right-to-left minima' and `number of right-to-left maxima' is symmetric.
\end{cor}

\subsection{Watermelon configurations}\label{sec:watermelon}

As a consequence of Theorem~\ref{thm:bottom-top-k}, we can recover a theorem of
Brak and Essam~\cite[Corollary 1]{BrakEssam2001} concerning certain families of $k$ non-intersecting paths called
\Dfn{watermelon configurations}.
\begin{definition}
  A \Dfn{watermelon configuration} of \Dfn{length} $x$ and
  \Dfn{deviation} $y$ is a family of $k$ non-intersecting lattice
  paths with northeast $(1,1)$ and southeast $(1,-1)$ steps,
  starting at $(0,2i)$ and terminating at $(x, y+2i)$ for
  $0\le i\le k-1$, not going below the $x$-axis.
\end{definition}
Brak and Essam derive the following statement using manipulations of a determinant.
\begin{theorem}[\cite{BrakEssam2001}]
  The number of watermelon configurations of length $x$ and
  deviation $y$ whose bottom path has $e$
  returns to the $x$-axis is the
  same as the number of families of $k$ non-intersecting paths where the
  lower $k-1$ paths form a watermelon configuration of length $x$ and
  deviation $y$, and the top path terminates at $(x-e-1, y+2k+e-3)$.
\end{theorem}
Christian Krattenthaler~\cite[Proposition 6]{Krattenthaler2006a}
gives a bijective proof by transforming the configurations into
certain semistandard Young tableaux and applying a variant of {\it jeu de
taquin}.  However, a more straightforward proof can be given by interpreting it as a special case of
Theorem~\ref{thm:bottom-top-k}.
\begin{proof}
  Any watermelon configuration of length $x$ and deviation $y$ can be
  transformed into a family $\Mat P$ of paths in $\P^k(T, B)$ with $T=N^{(x+y)/2} E^{(x-y)/2}$ and
  $B=(N E)^{(x-y)/2} N^y$, by letting all paths start at the origin and
  converting each northeast (southeast) step of a path in the
  watermelon configuration to a north (respectively east) step.

  Doing so, the returns to the $x$-axis of the bottom path of the watermelon
  configuration become the bottom contacts of the lower path in $\Mat P$, which are counted by $b(\Mat P)$.
  The proof of Theorem~\ref{thm:bottom-top-k} (or alternatively, Theorem~\ref{thm:bottom-left-top-right-k}) gives a bijection between tuples $\Mat P\in\P^k(T, B)$ with $b(\Mat P)=e$ and
  tuples $\Mat P\in\P^k(T, B)$ with $t(\Mat P)=e$.
  In the latter tuples, the upper path must have its $e$ top contacts at the end, so it is a path from the origin to
  $\big((x-y)/2-e, (x+y)/2-1\big)$, followed by the steps $NE^e$. Removing these forced steps from the corresponding watermelon configuration,
  we obtain a family as described in the statement.
\end{proof}
We remark that our bijection is different from Krattenthaler's.  One
might, however, speculate about a connection between {\it jeu de taquin}
and our theorem.

Even though Corollary~\ref{cor:easy-bottom} does not seem to
generalize nicely to families of paths, Krattenthaler gives
an explicit expression~\cite[Lemma~7]{Krattenthaler2006a} for the number of watermelon
configurations with a given number of bottom contacts, based on the
theorem above.  Via the Lindstr\"om--Gessel--Viennot method for
non-intersecting lattice paths, this amounts to a determinant
evaluation. It should also be possible to give explicit expressions
for the number of families of paths when the lower boundary has
arbitrary integer slope.

\section{Flagged semistandard Young tableaux and $k$-triangulations}\label{sec:SSYT}

In this section we discuss connections of
Theorem~\ref{thm:bottom-top-k}, which concerns $k$-tuples of non-crossing lattice paths, with two other combinatorial objects: flagged semistandard Young tableaux, and generalized triangulations
of a convex polygon.

A \Dfn{$k$-triangulation} of a convex $n$-gon is a maximal set of diagonals
such that no $k+1$ of them mutually cross.  In particular, a
$1$-triangulation is a triangulation in the usual sense.
Note that every $k$-triangulation contains all the diagonals between vertices at
distance $k$ or less, where distance is the number of sides of the polygon that separate the two vertices.
We call these \Dfn{trivial} diagonals, and we disregard them when computing the degree of a
vertex. The \Dfn{neighbors} of a vertex are then the vertices connected to it by a nontrivial diagonal.

It was shown in~\cite{MR1931939,MR1756126} that all $k$-triangulations of a convex $n$-gon have
the same number $k(n-2k-1)$ of nontrivial diagonals. Jonsson~\cite{Jonsson2005} proved that the
number of $k$-triangulations of an $n$-gon is given by the $k\times
k$ determinant of Catalan numbers $\det(C_{n-i-j})_{i,j=1}^k$.
Denoting by $\D_n^k$ the set $\P^k(T,B)$ where $T=N^{n-2k-1} E^{n-2k-1}$ and $B=(E N)^{n-2k-1}$, the above determinant is
also known to equal $|\D_n^k|$, by the Lindstr\"om--Gessel--Viennot method~\cite{GesselViennot1985}.
The elements of $\D_n^k$ are $k$-fans of Dyck paths of semilength $n-2k$ (or, in the terminology from Section~\ref{sec:watermelon},
watermelon configurations with deviation~$0$).
There are several well-known bijections between $1$-triangulations and Dyck paths.
For $k=2$, the first bijection between $2$-triangulations and pairs of non-crossing Dyck paths
was given in~\cite{Elizalde2006}. For general $k$, a bijection between $k$-triangulations of the $n$-gon and $\D_n^k$ has been recently found by 
Serrano and Stump~\cite{SerranoStump2010}.

Nicol\'as~\cite[Conjecture~1]{Nicolas2009} discovered
experimentally that the distribution of degrees of $k$ consecutive
vertices over the set of $k$-triangulations of a convex $n$-gon equals the
distribution of $(h_0,h_1,\dots,h_{k-1})$ over $\D_n^k$. This is stated as
Conjecture~\ref{conj:carlos} above.
Serrano and Stump's bijection~\cite[Theorem~4.4]{SerranoStump2010} proves
a special case of this conjecture: the degree of any given vertex in the set of
$k$-triangulations of a convex $n$-gon is equidistributed with the
number of top contacts of the upper path in $\D_n^k$.

To prove Conjecture~\ref{conj:carlos} in full generality, we construct a new bijection between tuples of paths
and certain semistandard Young tableaux. In the rest of this section, let $T=N^y E^x$ and let $B$ be a path from the origin to $(x,y)$, weakly below $T$ and ending with a north step.
The region enclosed by $T$ and $B$ can then be interpreted in an
obvious way as the Young diagram (in English notation) of a partition
$\lambda=\lambda(T,B)$. The parts of $\lambda$ are the lengths of the rows of
the diagram, from top to bottom, so that $x=\lambda_1\ge\lambda_2\ge\dots\ge\lambda_y\ge0$.

Recall that a filling of such a diagram with positive integer entries is
called a \Dfn{Young tableau of shape $\lambda$}, and that such a
tableau is called \Dfn{semistandard} (a \Dfn{SSYT} for short) if the entries in its rows are weakly
increasing and the entries in its columns are strictly increasing.  A
Young tableau (or simply tableau, from now on) is called \Dfn{$k$-flagged} if the entries of row $r$
are at most $k+r$ for every $r$.  Finally, the \Dfn{weight} of a
tableau is $\mu=(\mu_1,\mu_2,\dots)$ where $\mu_i$ is the
number of entries equal to $i$.

As in Section~\ref{sec:path-tuple}, given $\Mat
P=(P_1,P_2,\dots,P_k)\in\P^k(T, B)$, with the convention that $P_0 =
T$ and $P_{k+1} = B$, let $h_i(\Mat P)$, for $0\le i\le k$, be the
number of east steps where $P_i$ and $P_{i+1}$ coincide.  Also, let
$u_s(\Mat P)$, for $1\le s\le y-1$, be the number of east steps with
$y$-coordinate $y-s$ that lie strictly between $T$ and $B$ and are
not used by any of the paths $P_1,\dots,P_{k}$.

We can now state the main result of this section.
\begin{theorem}\label{thm:carlos}
  Let $T=N^y E^x$, and let $B$ be any path from the origin to $(x,y)$,
  weakly below $T$ and ending with a
  north step.  There is an explicit bijection $\bij$ between
  $k$-tuples of paths $\Mat P\in\P^k(T, B)$ with $h_i(\Mat P)=h_i$
  ($0\le i\le k$) and $u_s(\Mat P)=u_s$ ($1\le s<y$), and $k$-flagged
  SSYT of shape $\lambda=\lambda(T,B)$ and weight
  $$
  (\lambda_1-h_0, \lambda_1-h_1, \dots, \lambda_1-h_k,\,u_1, u_2,
  \dots, u_{y-1}).
  $$
\end{theorem}

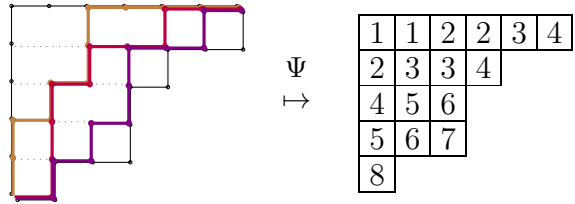
\begin{figure}[htb]
  \begin{center}
\begin{tikzpicture}[scale=0.5]
 \draw (-0.08,0.08) circle(1.2pt) \N\N\N\N\N\E\E\E\E\E\E;
 \draw (0.08,-0.08) circle(1.2pt) \E\N\E\E\N\N\E\N\E\E\N;
 \draw[very thick,brown] (-0.04,0.04) \n\n\e\n\e\n\n\e\e\e\e;
 \draw[very thick,purple] (0,0) \e\n\n\n\e\n\e\e\n\e\e;
 \draw[very thick,violet] (0.04,-0.04) \e\n\e\n\e\n\n\e\e\n\e;
 \draw[dotted,gray] (0,4)--(2,4);
 \draw[dotted,gray] (0,3)--(1,3);
 \draw[dotted,gray] (2,3)--(3,3);
 \draw[dotted,gray] (1,2)--(2,2);
 \draw[dotted,gray] (0,1)--(1,1);
 \draw (7.5,2.5) node [label=90:$\bij$] {$\mapsto$};
 \draw (12,2.5) node {\large\young(112234,2334,456,567,8)};
\end{tikzpicture}
\end{center}
  \caption{An example of the weight-preserving bijection $\bij$ from Theorem~\ref{thm:carlos}.
  The $3$-tuple on the left has $(h_0,h_1,h_2,h_3)=(4,3,3,3)$ and $(u_1,u_2,u_3,u_4)=(2,2,1,1)$.
  The $3$-flagged SSYT on the right has $\lambda_1=6$ and weight $(2,3,3,3,\,2,2,1,1)$.
  }\label{fig:bijectionSSYT}
\end{figure}

Figure~\ref{fig:bijectionSSYT} shows an example of the bijection $\bij$,
which will be described in Section~\ref{sec:proofcarlos}. The importance of $\bij$ lies in the particular way
that the statistics on tuples of paths determine the weight of the corresponding tableau. This property is what allows us in Section~\ref{sec:consequences} to use $\bij$ to prove and generalize Conjecture~\ref{conj:carlos}.
In fact, a much easier bijection between $k$-tuples of paths in $\P^k(T, B)$ and $k$-flagged SSYT of shape $\lambda(T,B)$ can be constructed as follows.
Given $\Mat P\in\P^k(T, B)$, first fill each cell of the Young diagram with the number of paths in $\Mat P$ that pass above the cell (this produces a reverse plane partition where all entries are at most $k$).
Then, for each row $r$, add $r$ to every entry in that row. Figure~\ref{fig:easybij} shows an example of this construction.
This bijection, which is used by Serrano and Stump~\cite{SerranoStump2010}, does not have the property described in Theorem~\ref{thm:carlos}, but it proves the following fact.

\begin{lemma}[\cite{SerranoStump2010}]\label{lem:easybij}
Let $T=N^y E^x$, and let $B$ an arbitrary path from the origin to $(x,y)$.
Then $|\P^k(T, B)|$ equals the number of $k$-flagged SSYT of shape $\lambda(T,B)$.
\end{lemma}

\begin{figure}[htb]
  \begin{center}
\begin{tikzpicture}[scale=0.5]
 \draw (-0.08,0.08) circle(1.2pt) \N\N\N\N\N\E\E\E\E\E\E;
 \draw (0.08,-0.08) circle(1.2pt) \E\N\E\E\N\N\E\N\E\E\N;
 \draw[very thick,brown] (-0.04,0.04) \n\n\e\n\e\n\n\e\e\e\e;
 \draw[very thick,purple] (0,0) \e\n\n\n\e\n\e\e\n\e\e;
 \draw[very thick,violet] (0.04,-0.04) \e\n\e\n\e\n\n\e\e\n\e;
 \draw[dotted,gray] (0,4)--(2,4);
 \draw[dotted,gray] (0,3)--(1,3);
 \draw[dotted,gray] (2,3)--(3,3);
 \draw[dotted,gray] (1,2)--(2,2);
 \draw[dotted,gray] (0,1)--(1,1);
 \draw (7.5,2.5) node {$\mapsto$};
 \draw (12,2.5) node {\large\young(112234,2245,355,567,6)};
\end{tikzpicture}
\end{center}
  \caption{An example of the non-weight-preserving bijection that proves Lemma~\ref{lem:easybij}.}\label{fig:easybij}
\end{figure}
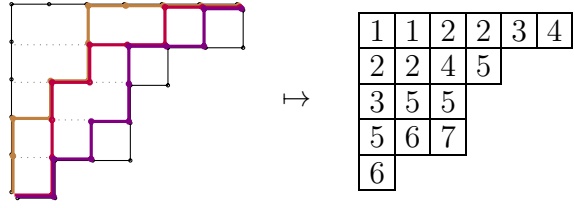

It may be surprising that the bijection $\bij$ only works for the special
top boundary of the region as stated in Theorem~\ref{thm:carlos}.  Indeed, one might be led to
believe that paths in an arbitrary region should be in bijection with
a suitable set of skew-semistandard Young tableaux.
We remark that our proof breaks in a subtle way in this setting.
It would be interesting to find such a generalization.

\subsection{Consequences}\label{sec:consequences}
Before we prove Theorem~\ref{thm:carlos}, let us first show how to infer Conjecture~\ref{conj:carlos}
from it.  To do so, we use a result of Serrano and Stump~\cite{SerranoStump2010}.
To simplify the statements of Theorems~\ref{thm:Serrano-Stump} and~\ref{thm:degrees}, we change the labeling of the vertices of the polygon, so that our vertex $i$ corresponds to
vertex $n+1-i$ in the notation of~\cite{SerranoStump2010}. In the next two theorems, $\T_n^k$ denotes the set of $k$-triangulations of a convex $n$-gon with vertices counterclockwise labeled from $1$ to $n$.
In such a $k$-triangulation, $d_i$ denotes the number of (nontrivial) neighbors of vertex $i$ among $i+1,i+2,\dots,n$.
Note that for $i\ge n-k$ we have $d_i=0$, and that for $1\le i\le k+1$, $d_i$ is the degree of vertex $i$, since all its neighbors are among $i+1,i+2,\dots,n$.
We call $(d_1,\dots,d_{n-k-1})$ the \Dfn{degree sequence} of the $k$-triangulation.
The next theorem is equivalent to~\cite[Corollary~3.5]{SerranoStump2010}. The description of the weight of the $k$-flagged tableau is not given explicitly in~\cite{SerranoStump2010},
but it is immediate from the Edelman-Greene insertion process that is used.

\begin{theorem}[\cite{SerranoStump2010}]\label{thm:Serrano-Stump}
 There is an explicit bijection between $k$-triangulations in $\T_n^k$ with degree sequence $(d_1,\dots,d_{n-k-1})$ and $k$-flagged SSYT of shape $(n-2k-1,n-2k-2,\dots,1)$ with weight $(\mu_1,\mu_2,\dots,\mu_{n-k-1})$, where
  \begin{equation*}
    \mu_i=
    \begin{cases}
      n-2k-1-d_i &\text{if }1\leq i\leq k+1,\\
      n-k-i-d_i &\text{if }k+1< i\leq n-k-1.
    \end{cases}
  \end{equation*}
\end{theorem}

Combining Theorems~\ref{thm:carlos} and~\ref{thm:Serrano-Stump}, we obtain the following
refinement of Conjecture~\ref{conj:carlos}. Note that this refinement gives a simple description of the distribution of the degrees $d_1,d_{2},\dots,d_{k+1}$ of $k+1$ consecutive vertices,
and not just $k$ as in the original conjecture. In particular, it implies that the joint distribution of $(d_1,d_{2},\dots,d_{k+1})$ over $\T_n^k$ is symmetric.

\begin{theorem}\label{thm:degrees}
  There is an explicit bijection between $k$-triangulations in $\T_n^k$ with degree sequence $(d_1,\dots,d_{n-k-1})$,
  and $k$-tuples in $\D_n^k$
  with parameters $(h_0,h_1,\dots,h_k)$ and $(u_1,u_2,\dots,u_{n-2k-2})$, where
$$
  d_{i}=
  \begin{cases} h_{i-1} &\text{if } 1\leq i\leq k+1,\\
  n-i-k-u_{i-k-1}& \text{if }k+1<i\le n-k-1.
  \end{cases}
 $$
\end{theorem}

\begin{proof}
The bijection from Theorem~\ref{thm:Serrano-Stump} sends $k$-triangulations of the $n$-gon to $k$-flagged SSYT of shape
$(n-2k-1,n-2k-2,\dots,1)$, which in turn become $k$-tuples in $\D_n^k$ applying the bijection $\bij$ from Theorem~\ref{thm:carlos}.
Comparing the weight of the tableau in both bijections, we see that for $1\leq i\leq k+1$, we have $n-2k-1-d_{i}=\mu_i=\lambda_1-h_{i-1}$, and since $\lambda_1=n-2k-1$, it follows that the degree of vertex $i$ is $d_{i}=h_{i-1}$. Similarly, for $k+1<i\le n-k-1$, we have $n-k-i-d_{i}=\mu_i=u_{i-k-1}$.
\end{proof}

Figure~\ref{fig:triang} shows an example of the composition of bijections in Theorem~\ref{thm:degrees}. In this example, $(d_1,\dots,d_5)=(1,2,2,0,1)$, $(\mu_1,\dots,\mu_5)=(2,1,1,2,0)$, $\lambda_1=3$, $(h_0,h_1,h_2)=(1,2,2)$, and $(u_1,u_2)=(0,1)$.

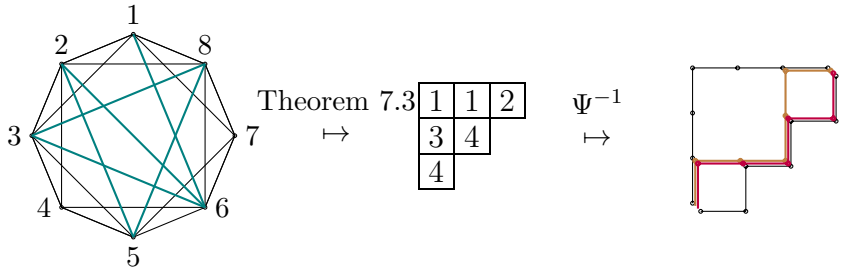
\begin{figure}[htb]
  \begin{center}
\begin{tabular}{cc}
\begin{tikzpicture}[scale=0.45]
      \draw (0,-3) coordinate(d4)
      -- (-2.12,-2.12) coordinate(d5)
      -- (-3,0) coordinate(d6)
      -- (-2.12,2.12) coordinate(d7)
      -- (0,3) coordinate(d8)
      -- (2.12,2.12) coordinate(d1)
      -- (3,0) coordinate(d2)
      -- (2.12,-2.12) coordinate(d3);
      -- (0,-3);
      \foreach \x in {1,...,8} {
        \draw (d\x) circle (1.5pt);
      }
      \draw (d1)--(d2); \draw (d2)--(d3); \draw (d3)--(d4); \draw (d4)--(d5); \draw (d5)--(d6); \draw (d6)--(d7); \draw (d7)--(d8); \draw (d8)--(d1);
      \draw (d1)--(d3); \draw (d2)--(d4); \draw (d3)--(d5); \draw (d4)--(d6); \draw (d5)--(d7); \draw (d6)--(d8); \draw (d7)--(d1); \draw (d8)--(d2);
      \draw[thick,teal] (d1)--(d4); \draw[thick,teal] (d1)--(d6); \draw[thick,teal] (d3)--(d6); \draw[thick,teal] (d3)--(d7); \draw[thick,teal] (d3)--(d8); \draw[thick,teal] (d4)--(d7);
      \draw (d8) node[above] {$1$};
      \draw (d7) node[above] {$2$};
      \draw (d6) node[left] {$3$};
      \draw (d5) node[left] {$4$};
      \draw (d4) node[below] {$5$};
      \draw (d3) node[right] {$6$};
      \draw (d2) node[right] {$7$};
      \draw (d1) node[above] {$8$};
      \draw (6,0) node [label=90:Theorem~\ref{thm:Serrano-Stump}]  {$\mapsto$};
      \draw (10,0) node {\large\young(112,34,4)};
\end{tikzpicture}
&
\begin{tikzpicture}[scale=0.6]
\draw (-2.2,1.5) node [label=90:$\bij^{-1}$]  {$\mapsto$};
 \draw (-0.09,0.09) circle(1.2pt) \N\N\N\E\E\E;
 \draw (0.09,-0.09) circle(1.2pt) \E\N\E\N\E\N;
 \draw[thick,brown] (-0.03,0.03) \N\E\E\N\N\E;
 \draw[thick,purple] (0.03,-0.03) \N\E\E\N\E\N;
  \draw(0,-1.5) ;
\end{tikzpicture}
\end{tabular}
\end{center}
  \caption{An example of the bijection from Theorem~\ref{thm:degrees}.}\label{fig:triang}
\end{figure}

We remark that $k$-flagged SSYT appear
frequently in the literature because the generating functions
of all such tableaux of a given shape with respect to their weight,
called $k$-flagged Schur functions, are
Schubert polynomials for certain vexillary permutations~\cite{SerranoStump2010,Wachs1985}.  In particular, we have the
following determinantal expression due to Michelle
Wachs~\cite[Theorem~3.5]{Wachs1985}.  An alternative proof using the
determinantal formula for the number of non-intersecting lattice
paths was given by Ira Gessel and Xavier
Viennot~\cite[Corollary~4]{GesselViennot1989}.
\begin{prop}[\cite{Wachs1985}]
  The generating function for $k$-flagged SSYT $S$ of shape $\lambda$ is
  \begin{equation*}
    \sum_S \Mat x^S =
    \det\left(h_{\lambda_i-i+j}(x_1,\dots,x_{k+i})\right)_{1\leq
      i,j\leq \ell},
  \end{equation*}
  where $\Mat x^S = \prod_i x_i^{\mu_i}$ and $(\mu_1,\mu_2,\dots)$ is the weight of $S$,
  $\ell$ is the number of parts of $\lambda$,
  and $h_n(x_1,\dots,x_{k+i})=\sum_{1\le j_1\le\dots\le j_n \le k+i}x_{j_1}\dots x_{j_n}$ is the complete homogeneous symmetric function.
\end{prop}

The above proposition implies that $k$-flagged Schur functions are symmetric in the variables $x_1,\dots,x_{k+1}$. Alternatively, this can be proved combinatorially in the same way
that one proves that Schur functions, defined as generating functions of SSYT, are symmetric~\cite[Theorem 7.10.2]{EC2}.
Using this symmetry, Theorem~\ref{thm:carlos} provides an alternative proof of Theorem~\ref{thm:bottom-top-k} in the case that $T=N^yE^x$.
Indeed, the first $k+1$ components of the weight have a symmetric joint distribution over $k$-flagged SSYT of a given shape $\lambda$, which, by Theorem~\ref{thm:carlos}, is the same distribution of $(\lambda_1-h_0,\dots,\lambda_1-h_k)$ over $\P^k(T,B)$, where $T$ and $B$ are the boundaries of the shape $\lambda$.  Theorem~\ref{thm:bottom-top-k} follows now immediately for such $T$ and $B$.

\subsection{Proof of Theorem~\ref{thm:carlos}}\label{sec:proofcarlos}
Throughout this section we assume that $T=N^yE^x$ and that $B$ is a path from the origin to $(x,y)$.
We have defined $k$-flagged tableaux without the requirement of
being semistandard because of a technical necessity that will become apparent
below. We also need a variation of
perforated tableaux as introduced by Georgia Benkart, Frank Sottile
and Jeffrey Stroomer~\cite{BenkartSottileStroomer1996}.
We say that an entry $e$ in a $k$-flagged tableau is \Dfn{small} (respectively \Dfn{large}) if $e\le k+1$ (respectively $e>k+1$).
An entry $e$ in row $r$ is called \Dfn{maximal} if $e=k+r$. Note that the entries in the first row of a $k$-flagged tableau are necessarily small.

\begin{definition}\label{dfn:kperflagged}
  A \Dfn{$k$-perflagged tableau} is a $k$-flagged Young tableau where any pair of entries $e_1,e_2$ with $e_2$ weakly southeast of $e_1$ satisfies the following conditions:
  \begin{itemize}
  \item If $e_1$ and $e_2$ are in the same row, and both are small or both are large, then $e_1\leq e_2$.
  \item If $e_1$ and $e_2$ are in different rows and both are large, then $e_1<e_2$.
  \item If $e_1$ and $e_2$ are in different rows and both are small, then $e_1\leq e_2$. Furthermore, if $e_1=e_2$, the following \Dfn{chain condition} must be met:

    Suppose that $e_1$ is in a cell $a$ in row $r$ and $e_2$ is in a cell
    $b$ in row $r+s$.  Then there is a sequence of cells $c_1,\dots,c_s$ (called a \Dfn{chain}) such that for $1\leq i\leq s$,
    cell $c_i$ is in row $r+i$ and weakly east of $c_{i-1}$ (with the convention that $c_0=a$), and the entry in $c_i$ is no larger
    than the entry in the cell just northeast of $c_i$.  Furthermore, $c_s$ is strictly west of $b$ (in the same row).
    In particular, $a$ and $b$ cannot be in the same column.
  \end{itemize}
\end{definition}

Intuitively, a $k$-flagged tableau is $k$-perflagged if, when restricting to large
(respectively small) entries, the semistandard condition (weakly increasing along rows and strictly increasing along columns) is met, with the caveat that a
small entry is allowed to be equal to another one strictly southeast of it in
some cases. When this happens, any small entries inside the minimal rectangle containing these two equal entries must also be equal to them.
For example, the $2$-perflagged tableau
$$\young(1111,2222,3444,4556,6663)$$
has two small entries equal to $3$ in different rows, one being southeast of the other. The chain condition is satisfied by taking $c_1$ to be the cell containing a $4$ in the first column,
and $c_2$ to be the cell containing a $6$ in the third column.

To prove Theorem~\ref{thm:carlos}, we construct a bijection $\bij$ in two parts.
First, we associate in a simple way a $k$-perflagged tableau
to each $k$-tuple of paths in $\P^k(T, B)$. This tableau has the weight claimed in Theorem~\ref{thm:carlos},
but in general it will not be semistandard.  Then, we modify the tableau using
local moves reminiscent of {\it jeu de taquin} to remove any occurring violation of the
semistandard property. It will be clear that the weight of the tableau is preserved by these operations.
However, proving that the operations are invertible and well defined will require some work.

\begin{definition}\label{def:tab}
  To each $k$-tuple of paths $\Mat P\in\P^k(T, B)$ we associate a
  Young tableau $\Tab(\Mat P)$ of shape $\lambda(T,B)$ as
  follows:
  \begin{enumerate}
  \item for each cell of the Young diagram whose upper boundary is an
    east step of a path $P_i$ (with $i\ge0$, where we again use the convention $P_0=T$), take the largest such $i$ and fill the cell with the number $i+1$;
  \item fill the remaining cells in each row $r$ with $k+r$.
  \end{enumerate}
\end{definition}

For an example of this construction, see the first step in
Figure~\ref{fig:exampletableaux}. Note that step~(i) fills the
cells with small entries, while step~(ii) fills the remaining cells with maximal
entries. It is clear that for any $\Mat P\in\P^k(T,
B)$, $\Tab(\Mat P)$ is a $k$-perflagged tableau. Indeed, rows are weakly increasing, and if $e_1$ is strictly north and weakly west of $e_2$ and entries $e_1,e_2$ are both large or both small, then $e_1<e_2$ (in particular, the chain condition is not needed).

\begin{definition}
  A \Dfn{semistandard violation} of a $k$-perflagged tableau $S$ is a cell
  containing an entry $e$ such that either
  \begin{itemize}
  \item there is a cell immediately above, whose entry $\ec$ satisfies $\ec\ge
    e$, or
  \item there is a cell immediately to the left, whose entry $\er$ satisfies
    $\er>e$.
  \end{itemize}
  The \Dfn{minimal semistandard violation} of $S$ is the one with the smallest
  entry, and among those with the same smallest entry, the one in the leftmost
  column.
\end{definition}
Note that, since $S$ is $k$-perflagged, a semistandard violation as defined
above can only happen if $e$ is a small entry, and either $\ec$ or $\er$ is
a large entry. In the $3$-perflagged tableau on the left of
Figure~\ref{fig:exampletableaux}, the four semistandard violations are marked by
diamonds, and the minimal one is the entry $3$ in the first column.

\begin{definition}
  A \Dfn{path violation} of a $k$-perflagged tableau $S$ is a cell containing a
  small entry $e$, such that either
  \begin{enumerate}
    \renewcommand{\labelenumi}{(\arabic{enumi})}
  \item\label{item:path-viol1} there is a cell immediately below containing a
    large entry which is not maximal, or
  \item\label{item:path-viol2} there is a cell immediately to the right containing a large entry,
  and all the small entries in this cell's column up to the top row are strictly less than~$e$.
  \end{enumerate}
  The \Dfn{maximal path violation} of $S$ is the one with the largest entry, and
  among those with the same largest entry, the one in the rightmost column.
\end{definition}
In the $3$-perflagged tableau on the right of
Figure~\ref{fig:exampletableaux}, the two path violations are circled, and the
maximal one is the entry $4$ in the third column.

Clearly, a $k$-perflagged tableau with no semistandard violations is a
$k$-flagged SSYT.  Now we show that $k$-perflagged tableaux with no path
violations correspond to $k$-tuples of paths via the construction in
Definition~\ref{def:tab}.

\begin{lemma}\label{lem:Tab}
  The map $\Mat P \mapsto \Tab(\Mat P)$ is a bijection between
  $\P^k(T, B)$ and the set of $k$-perflagged tableaux of shape
  $\lambda(T,B)$ with no path violations.  The weight of $\Tab(\Mat P)$ in terms of $h_i(\Mat P)$ and $u_s(\Mat P)$ is as in Theorem~\ref{thm:carlos}.
\end{lemma}
\begin{proof}
First we show that for any $\Mat P \in\P^k(T, B)$, the $k$-perflagged tableau $\Tab(\Mat P)$ has no path violations. By step (ii) in Definition~\ref{def:tab}, all large entries in row $r$ equal $k+r$,
and so $\Tab(\Mat P)$ has no path violations of type~\eqref{item:path-viol1}. To see that it has no path violations of type~\eqref{item:path-viol2}, suppose for contradiction that cell $c$ is such a violation, and let $e$ be the entry in $c$.
By construction, the upper boundary of $c$ is an east step of $P_{e-1}$. Since $c$ has a cell immediately to its right, the next east step of $P_{e-1}$ is not a bottom contact, and thus it is
the upper boundary of a cell containing a small entry $e'\ge e$, contradicting that $c$ is a path violation.

To show that the map $\Mat P \mapsto \Tab(\Mat P)$ is a bijection, we describe its inverse.
Given a $k$-perflagged tableau $S$ of shape $\lambda(T,B)$ with no path violations, construct an element of $\P^k(T,B)$ as follows.
For each $1\le j\le x$, consider the $j$-th column of the tableau, and suppose that the small entries of $S$ in this column are $e_1<e_2<\dots<e_s$. Note that they increase from top to bottom because $S$ is $k$-perflagged.
Now let the $j$-th east step of paths $P_1,P_2,\dots,P_{e_1-1}$ be the upper boundary of the cell containing $e_1$, and for each $2\le i\le s$, let
the $j$-th east step of $P_{e_{i-1}},P_{e_{i-1}+1},\dots,P_{e_i-1}$ be the upper boundary of the cell containing $e_i$. Finally, let
the $j$-th step of $P_{e_{s}},P_{e_{s}+1},\dots,P_{k}$ coincide with the $j$-th east step of $B$.

To see that the heights of the east steps of each path weakly increase from left to right, suppose the $j$-th step of path $P_i$ is the upper boundary of a cell $c$ containing the entry $e$. If there is no cell to the right of $c$,
then the $(j+1)$-st step of $P_i$ is at least as high as its $j$-th step. If there is such a cell, then the fact that $c$ is not a path violation guarantees that there is a small entry
of size at least $e$ in column $j+1$, in the same row as $c$ or above. Thus, the $(j+1)$-st step of $P_i$ is at least as high as its $j$-th step also in this case.
Thus, by adding north steps at the obvious places, this construction produces an an element $\Mat P\in\P^k(T,B)$.

Now we show that this $k$-tuple satisfies $\Tab(\Mat P)=S$. First, it is clear from the construction that the small entries of these tableaux agree. Second, the large entries of $S$ in row $r$ equal $k+r$, and so they agree with
$\Tab(\Mat P)$ as well.
Indeed, if $S$ had a large entry that is not maximal, then the entry immediately above it would be either a path violation of type~\eqref{item:path-viol1},
or another non-maximal large entry (since large entries in any given row strictly increase from top to bottom). In the second case, considering the entry immediately above and repeating
the argument would eventually lead to a path violation of type~\eqref{item:path-viol1}, since all the entries in the top row are small entries.
 It is also clear that $\Mat P$ is the unique $k$-tuple of paths satisfying $\Tab(\Mat P)=S$, because all the east steps of the paths are uniquely determined by the small entries in $S$.

Finally we prove that the weight is as claimed. Let us first consider large entries. A cell in row $r$ of $\Tab(\Mat P)$, where $2\le r\le y$ (recall that the first now has no large entries),
is filled with $k+r$ if and only if its upper boundary is an unused east step (i.e, it does not belong to any path in $\Mat P$) with $y$-coordinate $y-r+1$.
Regarding small entries, an entry $e$ with $1\le e\le k+1$ appears in a given column of $\Tab(\Mat P)$ if and only if the east steps of $P_{e-1}$ and $P_{e}$ in this column do not coincide.  Since the number of east steps of any path in $\Mat P$ is $\lambda_1$, the total number of cells filled with $e$ is $\lambda_1-h_{e-1}(\Mat P)$.
\end{proof}

We now define invertible, weight-preserving, local operations on
$k$-perflagged tableaux that \lq correct\rq\ their violations.

\begin{definition}
  Let $S$ be a $k$-perflagged tableau having its minimal semistandard violation
  at cell $c$. Let $e$, $\ec$, $\er$ be the entries in, above, and to the left
  of $c$, respectively. (If one of these cells is missing, we define the
  corresponding entry to be $0$.)  Define $j(S)$ to be the tableau obtained from
  $S$ by swapping $e$ and $\er$ if $\er > \ec$, and by swapping $e$ and $\ec$
  otherwise.
\end{definition}

In the rest of this section, we will use $c$ to denote the minimal semistandard violation of a $k$-perflagged tableau $S$, and we will use
$e$, $\ec$ and $\er$ to denote the entries in, above, and to the left of $c$, respectively.

\begin{definition}
  Let $S$ be a $k$-perflagged tableau having its maximal path violation at cell $d$.
  Let $f, f_{b}, f_{r}$ be the entries in, below and to the
  right of $d$, respectively. (In case one of these cells is missing, define the
  corresponding entry to be $0$).  Define $j^{-1}(S)$ to be the tableau obtained
  from $S$ by swapping $f$ with whichever of $f_{r}$ or
  $f_{b}$ is the smallest large entry, or, in case of a tie, with
  $f_{b}$.
\end{definition}

Note that in the above definition, at least one of $f_{r}$ (for a path violation of type~\eqref{item:path-viol2})
or $f_{b}$ (for type \eqref{item:path-viol1}) is a large entry.

The rest of the proof will proceed as follows. In Lemma~\ref{lem:jSperflagged} we show that the operation $j$ preserves the $k$-perflagged property. Lemma~\ref{lem:creation-of-violations} describes how
path and semistandard violations change when $j$ is applied. This and Lemma~\ref{lem:j-min} are needed to prove that the minimal semistandard violation of $S$ becomes the maximal path violation of $j(S)$ (Lemma~\ref{lem:maxmin}), which
is in turn used to show that $j^{-1}$ is indeed the inverse of $j$ (Lemma~\ref{lem:j-1}).

\begin{lemma}\label{lem:jSperflagged}
  Let $S$ be a $k$-perflagged tableau.  Then $j(S)$ is also a $k$-perflagged
  tableau.
\end{lemma}
Before proceeding to the proof let us remark that it is indeed
possible that in $j(S)$ there are two equal small entries, one
strictly southeast of the other, even if there is no such pair in
$S$.  This is demonstrated by the following example, where the
transformation $j$ is applied to a $1$-perflagged tableau:
$$
\young(1122,333,22,5)\quad\stackrel{j}{\mapsto}\quad\young(1122,233,32,5)
$$
\begin{proof}
Let $c$, $e$, $\ec$ and $\er$ be defined as above.
Using that $e$ is a small entry, $\max\{\er,\ec\}$ is a large entry, and small entries in $S$ are weakly increasing in rows and strictly increasing in columns, there are four possibilities for the relative values of these entries:
  \begin{enumerate}
  \item $\er\leq e\leq k+1 < \ec$ ($\er$ small),
  \item $e\leq k+1< \er\leq \ec$,
  \item $\ec< e\leq k+1 < \er$ ($\ec$ small),
  \item $e\leq k+1< \ec < \er$.
  \end{enumerate}

  Let us first consider cases (i) and (ii), in which $j$ makes the following local move:
  \begin{equation*}
    \young(:\ec,\er e)
    \quad\stackrel{j}{\mapsto}\quad
    \young(:e,\er\ec)
  \end{equation*}
  Since $S$ is $k$-flagged and $e$ is a small entry, it is clear that $j(S)$ is $k$-flagged. To check that $j(S)$ is $k$-perflagged, let us
  first consider relations between large entries where one is weakly southeast of the other. Since $\ec$ is the only large entry that is moved,
  we only need to check pairs involving the entry $\ec$ in cell $c$. There are two nontrivial cases for the other entry in the pair: it is either in the same row as $c$,
  or it is northwest of $c$ in the row immediately above. For large entries in other places, the relative position with respect to the moved entry $\ec$  is the same in $S$ and $j(S)$.

  \noindent {\it Large entries in the same row as $c$.} Since $S$ is $k$-perflagged, large entries to the right of $c$ must be strictly larger than $\ec$.
  Now consider entries to the left of $c$. In case (ii), $\er$ is large, so every large entry $f$ to its left satisfies $f\le\er\le\ec$.
  In case (i), if there was some large entry to the left of $\ec$, take the rightmost one. Then its right neighbor would be a small entry $g$ with $g\le\er\le e$,
  which would create a semistandard violation in $S$ contradicting the minimality of $c$.

  \noindent {\it Large entries northwest of $c$ one row above $c$.} We have to check that there
  is no entry equal to $\ec$. Suppose otherwise, and let $d$ be the rightmost cell with entry
  $\ec$.  It cannot have a small entry to its right, since this would be a
  smaller semistandard violation, so $d$ must be the cell immediately northwest of $c$.
  Thus, the entries around $c$ in $S$ are as follows:
  \begin{equation*}
    \young(\ec\ec,\er e)
  \end{equation*}
  In case (i), the cell containing $\er$ would be a semistandard violation, contradicting the minimality of $c$.
  In case (ii), we have two large entries in the same column with $\er\leq\ec$, contradicting that $S$ is $k$-perflagged.

  Next we consider relations between small entries in $j(S)$ where one is weakly southeast of the other. Let us first check pairs involving the moved entry $e$.
  Denote by $c_{\uparrow}$ the cell immediately above $c$, which contains $e$ in $j(S)$.
  There are three cases for the other entry in the pair: it is strictly northwest of $c_{\uparrow}$, it is in the same row as $c_{\uparrow}$, or it is strictly southeast of $c_{\uparrow}$.
  Recall that a $k$-perflagged tableau has no repeated small entries in the same column.

  \noindent {\it Small entries in the same row as $c_{\uparrow}$.} Since $S$ is $k$-perflagged, small entries to the left of $c_{\uparrow}$ must be
  less than or equal to $e$. If some entry to the right of $c_{\uparrow}$ was strictly smaller than $e$, take the leftmost one.
  The left neighbor in $S$ of this entry would then be a large entry (possibly $\ec$), and so it would create a semistandard violation smaller than $c$.

  \noindent {\it Small entries strictly northwest of $c_{\uparrow}$.} Since $S$ is $k$-perflagged, such an entry is either strictly less than $e$ or equal to $e$, in which case there is a chain in $S$ connecting it to $c$. The same chain without its last cell connects it to $c_{\uparrow}$ in $j(S)$.

  \noindent {\it Small entries strictly southeast of $c_{\uparrow}$.} Since $S$ is $k$-perflagged, these entries cannot be smaller than $e$.
  We only need to check that if there is an entry equal to $e$ in a cell $d$ strictly southeast of $c_{\uparrow}$, then there is a chain connecting $c_{\uparrow}$ and $d$.
  If the cell immediately to the right of $c_{\uparrow}$ also contains an entry $e$, then the chain connecting this entry to $d$ in $S$ is also a valid chain connecting $c_{\uparrow}$ to $d$ in $j(S)$:
  \begin{equation*}
    \Yvcentermath0\Yboxdim{14pt}%
    \young(:\ec e,\er e)\ddots\raisebox{-14pt}{\young(e)}
    \quad\stackrel{j}{\mapsto}\quad
    \young(:ee,\er\ec)\ddots\raisebox{-14pt}{\young(e)}
  \end{equation*}
  Otherwise, the entry $f$ immediately to the right of $c_{\uparrow}$ must be a large entry. Indeed, if it was small, the $k$-perflagged condition
  on $S$ would force $f\le e$ (since cell $d$ is weakly southeast of it), but then the cell containing $f$ would be a semistandard violation in $S$ smaller than $c$:
  \begin{equation*}
    \Yvcentermath0\Yboxdim{14pt}%
    \young(:\ec f,\er e)\ddots\raisebox{-14pt}{\young(e)}
    \quad\stackrel{j}{\mapsto}\quad
    \young(:ef,\er\ec)\ddots\raisebox{-14pt}{\young(e)}
  \end{equation*}
  Since $S$ is $k$-perflagged, we have that $\ec\leq f$, and $S$ contains a chain connecting the two entries $e$ in cells $c$ and $d$ (if $c$ and $d$ are in the same row, we consider this chain to be empty). Prepending the cell $c$ to the beginning of this chain, we obtain a chain in $j(S)$ connecting the two entries $e$ in cells $c_{\uparrow}$ and $d$.

  Finally, we have to check that for any pair of cells $d_1,d_2$ different from $c_\uparrow$ containing the same small entry $f$, with $d_2$ strictly southeast of $d_1$, there is a chain connecting $d_1$ and $d_2$ in $j(S)$.
  If $c_1,c_2,\dots,c_s$ is a chain in the notation from Definition~\ref{dfn:kperflagged}, we say that the chain \Dfn{passes through} the cells $c_i$ and the cells just northeast of each of the $c_i$'s.
  If $d_1$ and $d_2$ are connected by a chain in $S$ not passing through either of the cells $c$ and $c_\uparrow$, then the same chain connects them in $j(S)$. If the minimal rectangle containing $d_1$ and $d_2$
  contains also cell $c_\uparrow$, then $f=e$. Thus, concatenating a chain from $d_1$ to $c_\uparrow$ with a chain from $c_\uparrow$ to $d_2$, both of which have been shown to exist (allowing a chain to be empty for two cells in the same row), we get a chain from $d_1$ to $d_2$.
  The only situation where the rectangle spanned by $d_1$ and $d_2$ does not contain $c_\uparrow$ but a chain connecting $d_1$ and $d_2$ could pass through $c$ occurs if $d_1$ is in the same row as $c$, and $c$ is just northeast of a cell in the chain. In this case, the same chain connecting $d_1$ and $d_2$ in $S$ is a valid chain in $j(S)$, since the entry $e$ in $c$ has been replaced with a larger entry $\ec$.

\medskip

  Let us now turn to cases (iii) and (iv), in which $j$ makes the following local move:
  \begin{equation*}
    \young(:\ec,\er e)
    \quad\stackrel{j}{\mapsto}\quad
    \young(:\ec,e\er)
  \end{equation*}
  It is immediate that $j(S)$ is also $k$-flagged. To check that $j(S)$ is $k$-perflagged, let us
  consider large entries first. The only nontrivial relations that need to be checked are those involving the moved entry $\er$ and another large entry in the column it is moved to, the one containing cell $c$.
  Large entries below $c$ in its column must be strictly larger than $\er$, since $S$ is $k$-perflagged.
  Now consider large entries above $c$ in its column. In case (iii), there are none, because otherwise, the cell immediately below the bottommost such entry would be a semistandard violation smaller than $c$, using the fact that $\ec<e$.
  In case (iv), $\ec$ is large, so every large entry $f$ above it satisfies $f<\ec<\er$.

  Next we consider relations between small entries in $j(S)$ where one is weakly southeast of the other. Let us first check pairs involving the moved entry $e$. Let $c_{\leftarrow}$ be the cell immediately to the left of $c$, which contains $e$ in $j(S)$.

  \noindent {\it Small entries in the same column as $c_{\leftarrow}$.} Since $S$ is $k$-perflagged, small entries above $c_{\leftarrow}$ must be less than or equal to $e$. To show that they are in fact strictly smaller,
suppose that one of these entries is equal to $e$. Then the chain condition in $S$ applied to this entry and the entry $e$ in cell $c$ would require $\er\leq\ec$, contradicting our assumption.
  Also, all small entries below $c_{\leftarrow}$ in its column are strictly larger than $e$, because otherwise the topmost such entry would create a semistandard violation smaller than $c$, since $\er$ is large.

\noindent {\it Small entries strictly northwest of $c_{\leftarrow}$.} Since $S$ is $k$-perflagged, such an entry is either strictly less than $e$ or equal to $e$, in which case there is a chain in $S$ connecting it to $c$. Since $\er>\ec$, this chain does not contain $c_{\leftarrow}$, and so the same chain connects it to $c_{\leftarrow}$ in $j(S)$.

\noindent {\it Small entries strictly southeast of $c_{\leftarrow}$.} Since $S$ is $k$-perflagged, these entries cannot be smaller than $e$. If there is an entry equal to $e$, then the same chain in $S$ that connects it to cell $c$ is a chain in $j(S)$ that connects it to cell $c_{\leftarrow}$.

 Finally, we have to check that for any pair of cells $d_1,d_2$ different from $c_{\leftarrow}$ containing the same small entry $f$, with $d_2$ strictly southeast of $d_1$, there is a chain connecting $d_1$ and $d_2$ in $j(S)$.
The only nontrivial case is when the chain connecting $d_1$ and $d_2$ in $S$ passes through $c$ or $c_\leftarrow$. If the minimal rectangle containing $d_1$ and $d_2$
  contains also cell $c_\leftarrow$, then $f=e$, and so concatenating a chain from $d_1$ to $c_\leftarrow$ with a chain from $c_\uparrow$ to $d_2$ (allowing chains connecting cells in the same row to be empty) we get a chain from $d_1$ to $d_2$.
  If this rectangle does not contain $c_\leftarrow$, the only way for a chain connecting $d_1$ and $d_2$ to pass through $c$ occurs if $d_1$ is above $c$ in the same column, and $d_2$ is weakly southeast of $c$.
  But in this case $f<e$ (since they are in the same column in $S$) and $e\le f$ (since $d_2$ is weakly southeast of $c$ in $S$), which is a contradiction.
\end{proof}

In the following lemmas, $S$ is a $k$-perflagged tableau with minimal semistandard violation at cell $c$, containing entry $e$. We denote by
$j(c)$ the cell that this entry is moved to in $j(S)$.

\begin{lemma}\label{lem:creation-of-violations}
  Let $S$, $c$, $e$ and $j(c)$ as above. Denote by $\Set V_{path}^{\geq e}(S)$ the set of path violations of $S$ with entry at least $e$, and by $\Set V_{sstd}(S)$ the set of semistandard violations of
  $S$. Then
\begin{enumerate}  \renewcommand{\labelenumi}{(\alph{enumi})}
\item $\Set V_{path}^{\geq e}(j(S)) \setminus \Set V_{path}^{\geq e}(S) = \{j(c)\}$,
\item $\Set V_{sstd}(j(S)) \setminus \Set V_{sstd}(S) \subseteq \{j(c)\}$.
\end{enumerate}
\end{lemma}

\begin{proof}
To prove part~(a), suppose first that $\er\leq\ec$, which corresponds to cases (i) and (ii) in the proof of Lemma~\ref{lem:jSperflagged},
  where $j(c)$ is the cell above $c$. Clearly, $j(c)$ is a path violation of type~\eqref{item:path-viol1} in $j(S)$, because it contains a small entry $e$, and below it there is a large entry $\ec$ which is not maximal, as it occurs in $S$ in a higher row than in $j(S)$. To see that $j(S)$ has no other path violations aside from the ones in $S$, it is enough to check that the cell to the left of $c$, which contains $\er$, is not a path violation of type~\eqref{item:path-viol2}.
  For this to happen, $\er$ would have to be a small entry (case (i)), but then the fact that the next column to its right contains a small entry $e$ satisfying $\er\le e$ prevents it from being a path violation.

  Suppose now that $\ec < \er$ (cases (iii) and (iv)), where $j(c)$ is the cell to the left of $c$. Now $j(c)$ is a path violation of type~\eqref{item:path-viol2} in $j(S)$, because
  it contains a small entry $e$, there is a large entry $\er$ to its right, and all the small entries above $c$ in $j(S)$ are strictly less than $e$, since $S$ is $k$-perflagged.
  The only possible additional path violation of $j(S)$ would be the cell above $c$ containing $\ec$, if this was a small entry. But then $\ec<e$, and so this cell is not in $\Set V_{path}^{\geq e}\big(j(S)\big)$.

  It remains to prove part~(b). In both cases, whether the local move is
  \begin{equation*}
    \young(:b,a\ec,\er e)\quad\stackrel{j}{\mapsto}\quad \young(:b,ae,\er\ec) \qquad\qquad \mbox{or} \qquad\qquad
    \young(:b,a\ec,\er e) \quad\stackrel{j}{\mapsto}\quad \young(:b,a\ec,e\er),
  \end{equation*}
  a semistandard violation is created at $j(c)$ when $a$ or $b$ are large entries.
  It is clear that no other semistandard violations are created.
\end{proof}

We remark that the above lemma implies that if $j(c)$ is a semistandard violation of $j(S)$, then it is minimal. Indeed, by part~(b), the tableau $j(S)$ has no semistandard violations other than $j(c)$ and
those in $S$. Since the cell $j(c)$ is either in the same column as $c$ or to its left, and it contains the same entry as the minimal semistandard violation of $S$, it has to be minimal in $j(S)$.

\begin{lemma}\label{lem:j-min}
 Let $S$, $c$, $e$, and $j(c)$ as above, and let $d$ be the minimal semistandard violation of $j(S)$.
Suppose that $d\neq j(c)$ and that $d$ also contains the entry $e$. Then $d$ is strictly to the right of~$c$.
\end{lemma}
\begin{proof}
The entry $e$ in cell $d$ does not move when $j$ is applied to $S$, since $j$ switches the entry $e$ in cell $c$ (moving it to to $j(c)$) with a strictly larger entry.
Now, since $c$ and $d$ contain the same entry $e$ in $S$, which is a $k$-perflagged tableau, they have to be in different columns.
But if $d$ was to the left of $c$, then $d$ would be a semistandard violation of $S$ contradicting the minimality of $c$. Thus, $d$ is strictly to the right of $c$.
\end{proof}

In the following lemma, $j^m$ denotes the composition of $j$ with itself $m$ times.

\begin{lemma}\label{lem:maxmin}
Let $R$ be a $k$-perflagged tableau with no path violations, and let $S=j^m(R)$ for some $m\ge0$. If $c$ is the minimal semistandard violation of $S$, then $j(c)$ is the maximal path violation of $j(S)$.
\end{lemma}

\begin{proof}
We proceed by induction on $m$. For $m=0$, this follows from Lemma~\ref{lem:creation-of-violations}(a), since $\Set V_{path}^{\geq e}(j(S))=\{j(c)\}$ in this case.

Now suppose that the statement is true for $m$, that is, $c$ is the minimal semistandard violation of $S = j^m(R)$, and $j(c)$ is the maximal path violation of $j(S)$.
Let $d$ be the minimal semistandard violation of $j(S)$. To prove that the statement holds for $m+1$, we have to show that the maximal path violation of $j^2(S)$ is $j(d)$.

Let $e$ be the entry in cell $c$ of $S$, and let $f$ be the entry in cell $d$ of $j(S)$.
Note that $f\ge e$, because otherwise, by Lemma~\ref{lem:creation-of-violations}(b), $d$ would have been a semistandard violation of $S$ smaller than $c$.
Suppose for contradiction that the maximal path violation of $j^2(S)$ is some cell $a$ other than $j(d)$, and let $g$ be its entry.
Lemma~\ref{lem:creation-of-violations}(a) applied to $j(S)$ states that $$\Set V_{path}^{\geq f}(j^2(S)) \setminus \Set V_{path}^{\geq f}(j(S)) = \{j(d)\}.$$
It follows that $j(d)$ is a path violation of $j^2(S)$, and that $a$ is a path violation of $j(S)$, since $g\ge f$ (because $a$ is maximal in $j^2(S)$).
If $g>e$, then $a$ would be a path violation of $j(S)$ larger than $j(c)$, a contradiction.
If $g=f=e$, then $a$ must be strictly to the right of $j(d)$, since it is maximal.
But $a$ and $d$ cannot be in the same column, since they are different cells containing the same entry in $j(S)$, and so $a$ is strictly to the right of $d$ as well.
If $d\neq j(c)$, then by Lemma~\ref{lem:j-min}, $d$ is strictly to the right of $c$ and $j(c)$. In all cases, $a$ is strictly to the right of $j(c)$, which contradicts the fact that $j(c)$ is the maximal path violation of $j(S)$.
\end{proof}

\begin{lemma}\label{lem:j-1}
Let $R$ be a $k$-perflagged  tableau with no path violations, and let $S=j^m(R)$ for some $m\ge0$. Suppose that $S$ has some semistandard violation. Then $$j^{-1}(j(S))=S.$$
\end{lemma}
\begin{proof}
Let $c$, $e$, $\er$ and $\ec$ as before. Suppose first that $\er\leq \ec$. If there is a cell just northeast of $c$, let $\erb$ be its entry.
Then $j$ makes the following local move when applied to $S$ (the cell with $\erb$ may not be there):
  \begin{equation*}
    \young(:\ec\erb,\er e)
    \quad\stackrel{j}{\mapsto}\quad
    \young(:e\erb,\er\ec)
  \end{equation*}
By Lemma~\ref{lem:maxmin}, cell $j(c)$ is the maximal path violation of $j(S)$.
If $\erb$ is a large entry, then $\ec \leq \erb$, since $S$ is $k$-perflagged. In any case, when $j^{-1}$ is applied to $j(S)$, it switches the entry $e$ in cell $j(c)$ with the entry $\ec$ in cell $c$.

Suppose now that $\ec < \er$. If there is a cell just southwest of $c$, let $\ecb$ be its entry. Now $j$ makes the following local move (the cell with $\ecb$ may not be there):
  \begin{equation*}
    \young(:\ec,\er e,\ecb)
    \quad\stackrel{j}{\mapsto}\quad
    \young(:\ec,e\er,\ecb)
  \end{equation*}
Again by Lemma~\ref{lem:maxmin}, $j(c)$ is the maximal path violation of $j(S)$.
If $\ecb$ is a large entry, then $\er< \ecb$, since $S$ is $k$-perflagged. Thus, when $j^{-1}$ is applied to $j(S)$, it switches the entry $e$ in cell $j(c)$ with the entry $\er$ in cell $c$.
\end{proof}

\begin{proof}[Proof of Theorem~\ref{thm:carlos}.]
Let $\Mat P\in \P^k(T, B)$ with $h_i(\Mat P)=h_i$ and $u_s(\Mat P)=u_s$.
By Lemma~\ref{lem:Tab}, the map $\Mat P \mapsto \Tab(\Mat P)$ is a bijection
between such $k$-tuples of paths and $k$-perflagged tableaux of shape $\lambda(T,B)$ with no path violations
and weight $(\lambda_1-h_0, \lambda_1-h_1, \dots, \lambda_1-h_k,\,u_1, u_2, \dots, u_{y-1})$.
To transform such a tableau into a SSYT, we repeatedly apply $j$, until no semistandard violation occurs.
To see that this process ends in a finite number of steps,
define a function that associates to each tableau the positive integer $\sum_{i,j}(i+j)e_{ij}$, where
$e_{ij}$ is the entry in row $i$ from the top and column $j$ from the left, and the sum is over all cells in the tableau.
This function strictly decreases each time that $j$ is applied, so the process ends in a tableau with no semistandard violations, which we define as $\bij(\Mat P)$.
Since $j$ is clearly weight-preserving, $\bij(\Mat P)$ is a SSYT of shape $\lambda(T,B)$ with the weight as claimed.

It remains to show that $\bij$ is a bijection. By Lemma~\ref{lem:j-1}, the process that transforms a $k$-perflagged tableaux of shape $\lambda(T,B)$ with no path violations
into a $k$-flagged SSYT is reversible. Thus, when disregarding the statistics $h_i$ and $u_s$, our map $\bij$ extends to an injection from the set of all $k$-tuples in $\P^k(T,B)$ to the set of all $k$-flagged SSYT of shape $\lambda(T,B)$.
Since these two sets have the same cardinality by Lemma~\ref{lem:easybij}, the map $\bij$ is surjective as well.
\end{proof}

The example in Figure~\ref{fig:exampletableaux} illustrates the bijection $\bij$, starting from a $k$-tuple of paths, constructing a $k$-perflagged tableau with no path violations, and then repeatedly applying $j$ to obtain a $k$-perflagged SSYT.

\begin{figure}[htb]
  \def\ct{\circled{3}}
  \def\cbt{\thickcircled{3}}
  \def\cf{\circled{4}}
  \def\cbf{\thickcircled{4}}
  \def\dt{\diamonded{3}}
  \def\dbt{\thickdiamonded{3}}
  \def\df{\diamonded{4}}
  \def\dbf{\thickdiamonded{4}}
  \def\cdbt{\thickcirclediamonded{3}}
  \Yboxdim{15pt}%
\begin{center}
  \begin{tikzpicture}[scale=0.53]
 \draw (-0.08,0.08) circle(1.2pt) \N\N\N\N\E\E\E\E;
 \draw (0.08,-0.08) circle(1.2pt) \E\E\N\E\N\E\N\N;
 \draw[very thick,brown] (-0.04,0.04) \n\n\n\e\n\e\e\e;
 \draw[very thick,purple] (0,0) \n\e\n\e\e\n\e\n;
 \draw[very thick,violet] (0.04,-0.04) \e\n\n\e\e\e\n\n;
 \draw (4.4,2) node[right] {$\stackrel{\Tab}{\mapsto}
    \young(1222,255\ct,6\cf\cf,\cbt7)\stackrel{j}{\mapsto}
    \young(1222,255\cbt,\dbt\cf\cf,67)\stackrel{j}{\mapsto}
    \young(1222,25\cdbt5,\dt\cf4,67)\stackrel{j}{\mapsto}
    \young(1222,2\dbt55,\dt4\cbf,67)\stackrel{j}{\mapsto}
    \young(1222,23\dbf5,\dt45,67)$};
\end{tikzpicture}
\end{center}
\caption{The sequence of tableaux in the construction of $\bij(\Mat P)$ for a $3$-tuple of paths $\Mat P$. The (minimal) semistandard violations are indicated by (bold)
  diamonds and the (maximal) path violations by (bold) circles.}
\label{fig:exampletableaux}
\end{figure}
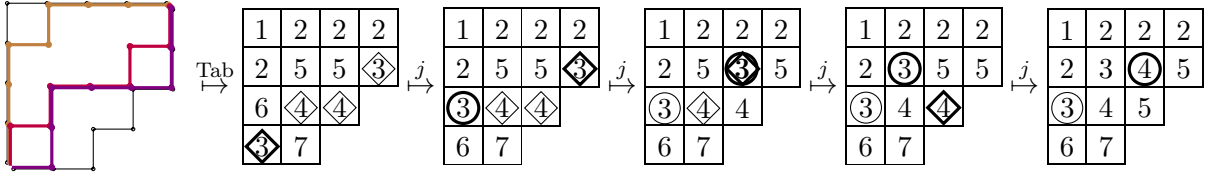

\appendix
\section{A bijective proof of the independence of the Tutte
  polynomial from the ordering of the ground set.}
\label{sec:Tutte-independence}

In Section~\ref{sec:bottom-left-top-right} we mentioned
that we do not know of a \lq natural\rq\
bijective proof of Theorem~\ref{thm:bottom-left-top-right}.
In the language of lattice path matroids, this would be a bijection from $\B$ to itself that turns internal and external
activity with respect to the ordering $1\prec2\prec3\prec\cdots$ of the ground set
into internal and external activity with respect the ordering $\cdots\prec3\prec2\prec1$.
However, in the basic case that the ordering of the ground set is modified
only by transposing two adjacent elements,
we can give a bijection from the set of bases of \emph{any} matroid to itself that
preserves internal and external activity. Since we can go from any ordering of the ground set to any other ordering by successive transpositions of
adjacent elements, a bijection proving Theorem~\ref{thm:bottom-left-top-right} is obtained by composing $\binom{|E|}{2}$ iterations of the bijection below, where $E$ is the ground set of the matroid.

Since this appendix applies to an arbitrary matroid, we discontinue the
typographic conventions of the other sections. In this section, $E$ denotes the ground set, and the letters
$a$, $b$, $c$, $d$, $e$, $x$ and $y$ are used to denote elements of $E$, while $\BB$ denotes the set of bases,
and the letters $B$, $B'$, $C$ and $D$ are used to denote bases. We use $B-x+y$ to denote $B\setminus\{x\}\cup\{y\}$. Furthermore, if $B$ contains exactly one element of the pair $\{x,y\}$,
we define $B^{x\leftrightarrow y}$ to be $B-x+y$ if $x\in B$, or $B-y+x$ otherwise (note that $B^{x\leftrightarrow y}=B^{y\leftrightarrow x}$ by definition).
If $B\in\BB$, recall that $x$ is \Dfn{active} with respect to $(B,\prec)$ if it cannot be switched with a smaller element to produce another base, that is, there
is no $y\prec x$ such that $B^{x\leftrightarrow y}$ is defined and belongs to $\BB$.

In the proof of the following theorem we use the \Dfn{strong basis exchange property} of matroids:
\begin{lemma}\label{lem:bep}
Let $C$ and $D$ be bases of a matroid, and let $d\in D\setminus C$. Then there exists $c\in C\setminus D$ such that $C-c+d$ and $D-d+c$ are bases of the matroid.
\end{lemma}

\begin{theorem}
  Let $\prec$ be a linear order on $E$, and let $x\prec y$ be adjacent elements in this order.
  Let $\prec'$ be the order obtained from $\prec$ by reversing the relative order of $x$ and $y$ and keeping the rest of order relationships unchanged.
  Define a map $\varphi_{xy}:\BB\rightarrow\BB$ by letting $\varphi_{xy}(B)=B'$, where
\beq\label{eq:B'}B'=\begin{cases} B^{x\leftrightarrow y} & \text{if $B$ contains exactly one element from $\{x,y\}$, $B^{x\leftrightarrow y}\in\BB$, and}\\
& \text{either $x$ is active w.r.t.\ $(B,\prec)$ or $y$ is active w.r.t.\ $(B,\prec')$;}\\
B &\text{otherwise.}
\end{cases}\eeq
Then $\varphi_{xy}$ is a bijection with the property that
the internal and external activity of $(B,\prec)$ equal the internal and external activity of $(B',\prec')$, respectively.
\end{theorem}

\begin{proof}
Let us first show that $\varphi_{xy}$ is a bijection. We claim that given $B'\in\BB$, we can recover $B$ applying the map $\varphi_{yx}$, which sends $B'$ to
\beq\label{eq:B''}B''=\begin{cases} B'^{y\leftrightarrow x} & \text{if $B'$ contains exactly one element from $\{x,y\}$, $B'^{y\leftrightarrow x}\in\BB$, and}\\
& \text{either $y$ is active w.r.t.\ $(B',\prec')$ or $x$ is active w.r.t.\ $(B',\prec)$;}\\
B' &\text{otherwise.}
\end{cases}\eeq
The fact that $B''=B$ is clear if either $x,y\in B$ or $x,y\notin B$, or if $B^{x\leftrightarrow y}\notin\BB$, because in these cases $B''=B'=B$.
Suppose now that $B$ contains exactly one element from $\{x,y\}$, and $B^{x\leftrightarrow y}\in\BB$. If
$x$ is inactive w.r.t.\ $(B,\prec)$ and $y$ is inactive w.r.t.\ $(B,\prec')$, then $B'=B$ by definition, and in this case, since
$y$ is inactive w.r.t.\ $(B',\prec')$ and $x$ is inactive w.r.t.\ $(B',\prec)$, we have that $B''=B$.

The only remaining case is when the conditions in the first part of~\eqref{eq:B'} hold, and so $B'=B^{x\leftrightarrow y}$.
In this case, the following two statements are true:
\ben
\item $x$ is active w.r.t.\ $(B,\prec)$ if and only if $y$ is active w.r.t.\ $(B',\prec')$,
\item $y$ is active w.r.t.\ $(B,\prec')$ if and only if $x$ is active w.r.t.\ $(B',\prec)$.
\een
Statement (i) is clear since $B^{x\leftrightarrow e} = B'^{y\leftrightarrow e}$ for every $e$ for which these are defined, and $e\prec x$ if and only if $e\prec' y$.
Similarly, statement (ii) holds because $B^{y\leftrightarrow e} = B'^{x\leftrightarrow e}$ for every $e$ for which these are defined, and $e\prec' y$ if and only if $e\prec x$.
Now we show that the conditions in the first part of~\eqref{eq:B''} are satisfied, and thus $B''=B'^{y\leftrightarrow x}=B$.
Indeed, since $B$ contains exactly one element from $\{x,y\}$, so does $B'=B^{x\leftrightarrow y}$, and we have $B'^{y\leftrightarrow x}=B\in\BB$.
Additionally, since either $x$ is active w.r.t.\ $(B,\prec)$ or $y$ is active w.r.t.\ $(B,\prec')$, the statements (i) and (ii) imply that either $y$ is active w.r.t.\ $(B',\prec')$ or $x$ is active w.r.t.\ $(B',\prec)$.

Next we show that the internal and external activity of $(B,\prec)$ equal the internal and external activity of $(B',\prec')$, respectively.
Consider first the case that $B'=B$, which happens if any of the following hold: \ben \renewcommand{\labelenumi}{(\alph{enumi})}
\item either $x,y\in B$ or $x,y\notin B$; \item $B^{x\leftrightarrow y}\notin\BB$; \item  $x$ is inactive w.r.t.\ $(B,\prec)$ and $y$ is inactive w.r.t.\ $(B,\prec')$. \een
In all three subcases, it is clear that each $e\notin\{x,y\}$ is active w.r.t.\ $(B,\prec)$ if and only if it is active w.r.t.\ $(B,\prec')$, since the relative order of $e$ with the other elements of $E$ does not change. Let us now show that for $e\in\{x,y\}$, $e$ is equally active w.r.t to both orderings. If (a) and (b) hold, this is clear, because $x$ and $y$ cannot be switched with each other to produce a base.
In case (c), the reason is that $x$ is inactive w.r.t.\ $(B,\prec)$, so it is also inactive w.r.t.\ $(B,\prec')$, since the elements smaller than $x$ in $\prec$ are still smaller than $x$ in $\prec'$;
similarly, $y$ is inactive w.r.t.\ $(B,\prec')$, so it is inactive w.r.t.\ $(B,\prec)$ as well.

Finally, consider the case in which none of (a), (b) and (c) hold, and so $B'=B^{x\leftrightarrow y}$. In this case, $y$ is always inactive w.r.t.\ $(B,\prec)$ because $x\prec y$ and $B^{x\leftrightarrow y}\in\BB$.
Similarly, $x$ is always inactive w.r.t.\ $(B',\prec')$ because $y\prec' x$ and $B'^{y\leftrightarrow x}=B\in\BB$. On the other hand, by statement (i) above,
$x$ is active w.r.t.\ $(B,\prec)$ if and only if $y$ is active w.r.t.\ $(B',\prec')$. Thus, we have proved that the number of active elements among $\{x,y\}$ is the same w.r.t.\ both $(B,\prec)$ and $(B',\prec')$.

Next we show that any $e\notin\{x,y\}$ is active w.r.t.\ $(B,\prec)$ if and only if it is active w.r.t.\ $(B',\prec')$.
It is enough to show that such an $e$ cannot be inactive w.r.t.\ $(B, \prec)$ but active w.r.t.\ $(B', \prec')$. The symmetric statement then follows
from the fact that $\varphi_{yx}(B')=B$.

Suppose for contradiction that $e\notin\{x,y\}$ is inactive w.r.t.\ $(B, \prec)$ but active w.r.t.\ $(B', \prec')$.
By definition, there is an element $a$ such that $a\prec e$, and $B^{e\leftrightarrow a}$ is defined and is a base.
We can easily discard the case that $a\in\{x,y\}$, because letting $b$ be such that $\{a,b\}=\{x,y\}$, we would have $B^{e\leftrightarrow a}=B'^{e\leftrightarrow b}\in\BB$ and $b\prec' e$,
so $e$ would be inactive w.r.t.\ $(B', \prec')$. Thus we assume that $a\notin\{x,y\}$.
For simplicity of notation, we suppose in what follows that $x\in B$, and so $B'=B-x+y$ (if $y\in B$ instead, just replace all $x$ with $y$ and all $y$ with $x$).

If $e\notin B$, then $B^{e\leftrightarrow a}=B-a+e$.
By Lemma~\ref{lem:bep} applied to the bases $C=B'$ and $D=B-a+e$ with $d=e$, there is a $c\in C\setminus D=\{y,a\}$ such that $B'-c+e$ and $D-e+c=B-a+c$ are bases.
If $c\prec'e$, then the fact that $B'-c+e\in\BB$ contradicts the assumption that $e$ is
active w.r.t.\ $(B',\prec')$. Thus, we must have $e\prec'c$, which implies that $c=y$. But then, $B'-y+e=B-x+e\in\BB$ and $e\prec x$, so $x$ is inactive w.r.t.\ $(B,\prec)$. Also, $B-a+y\in\BB$
and $a\prec' e\prec' y$, so $y$ is inactive w.r.t.\ $(B,\prec')$. But then (c) would hold, contradicting our assumption.

If $e\in B$, then $B^{e\leftrightarrow a}=B-e+a$ and the argument is very similar:
we apply Lemma~\ref{lem:bep} with $C=B-e+a$, $D=B'$ and $d=e$ to conclude that
either $e$ is active w.r.t.\ $(B',\prec')$, or $y$ is inactive w.r.t.\ $(B,\prec')$ and $x$ is inactive w.r.t.\ $(B,\prec)$, reaching a contradiction in both cases.
\end{proof}

\bibliographystyle{amsplain}

\providecommand{\cocoa} {\mbox{\rm C\kern-.13em o\kern-.07em C\kern-.13em
  o\kern-.15em A}}\def\cprime{$'$}
\providecommand{\bysame}{\leavevmode\hbox to3em{\hrulefill}\thinspace}
\providecommand{\MR}{\relax\ifhmode\unskip\space\fi MR }
\providecommand{\MRhref}[2]{%
  \href{http://www.ams.org/mathscinet-getitem?mr=#1}{#2}
}
\providecommand{\href}[2]{#2}

\end{document}